\documentclass{article}
\usepackage{silence}
\usepackage[utf8]{inputenc}
\usepackage[margin=.99in]{geometry}
\usepackage{xcolor}
\usepackage{etoolbox}
\usepackage{graphicx}
\usepackage{bm}
\usepackage{bbm}
\usepackage{amsmath}
\usepackage{amssymb}
\usepackage{amsfonts}
\usepackage{mathtools}
\usepackage{xparse}
\usepackage{epstopdf}
\usepackage[colorlinks=true,citecolor=blue]{hyperref}
\usepackage[capitalise]{cleveref}

\usepackage{setspace}
\DeclareDocumentCommand\littleo{s o m} {o\IfNoValueF{#2}{_{#2}}\IfBooleanTF{#1}{(#3)}{\left(#3\right)}}
\DeclareDocumentCommand\bigo{s o m} {O\IfNoValueF{#2}{_{#2}}\IfBooleanTF{#1}{(#3)}{\left(#3\right)}}
\DeclareDocumentCommand\sobolev{m m o} {H^{#1}_{#3}(#2 \IfNoValueF{#3})}
\DeclareDocumentCommand\lp{m m o} {L^{#1}\IfNoValueF{#3}{_{#3}}\left(#2\right)}
\DeclareDocumentCommand\cont{o m o} {C\IfNoValueF{#1}{^{#1}}(#2\IfNoValueF{#3}{;#3})}
\DeclareDocumentCommand\contc{o m o} {C_c\IfNoValueF{#1}{^{#1}}(#2\IfNoValueF{#3}{;#3})}
\DeclareDocumentCommand\norm{s m o} {\IfBooleanTF{#1}{\|#2\|}{\left\|#2\right\|}\IfNoValueF{#3}{_{#3}}}
\DeclareDocumentCommand\seminorm{s m o} {\IfBooleanTF{#1}{\|#2\|}{\left\|#2\right\|}\IfNoValueF{#3}{_{#3}}}
\DeclareDocumentCommand\ip{s m m o} {\IfBooleanTF{#1}{( #2,#3 )}{\left( #2,#3 \right)}\IfNoValueF{#4}{_{#4}}}
\DeclareDocumentCommand\eip{s m m o} {\IfBooleanTF{#1}{\langle #2,#3 \rangle}{\left\langle #2,#3 \right\rangle}\IfNoValueF{#4}{_{#4}}}
\DeclareDocumentCommand\abs{s m o} {\IfBooleanTF{#1}{|#2|}{\left|#2\right|}\IfNoValueF{#3}{_{#3}}}
\DeclareDocumentCommand\eucnorm{s m o} {\IfBooleanTF{#1}{|#2|}{\left|#2\right|}\IfNoValueF{#3}{_{#3}}}
\DeclareDocumentCommand\wass{s o m m} {W_{\IfNoValueF{2}{#2}}\IfBooleanTF{#1}{(#3, #4)}{\left(#3, #4\right)}}

\DeclareMathOperator{\e}{e}

\DeclareMathOperator{\hessian}{\mathrm D^2\!}
\DeclareMathOperator*{\argmin}{arg\,min}

\newcommand{\vect}[1]{\mathbf{#1}}

\newcommand{\dummy}{\mathord{\color{black!33}\bullet}}%
\newcommand{\expect}{\mathbf{E}}
\newcommand{\field}{\mathbf}
\newcommand{\statespace}{\mathbf K}
\newcommand{\grad}{\vect\nabla}

\newcommand{\laplacian}{\triangle}
\newcommand{\nat}{\mathbf N}

\newcommand{\proba}{\mathbf P}
\newcommand{\real}{\field R}
\renewcommand{\t}{\mathsf T}
\newcommand{\torus}{\field T}
\renewcommand{\d}{\mathrm d}

\makeatletter
\DeclareDocumentCommand \derivative{s m o m}{%
    \def\@der{\IfBooleanTF{#1}{\mathrm{d}}{\partial}}
    \def\@default{%
        \mathchoice{%
                \frac{%
                    \@der\ifnum\pdfstrcmp{#2}{1}=0\else^{#2}\fi {\IfNoValueTF{#3}{}{#3}}
                }{%
                    \@for\@token:={#4}\do{\@der \@token}
                }
            } {%
                \@for\@token:={#4}\do{\@der_{\@token}\ifnum\pdfstrcmp{#2}{1}=0\else^{#2}\fi} \IfNoValueTF{#3}{}{#3}
            } {} {}
    }
    \IfBooleanTF{#1}{\@default}{\@default}
}
\makeatother

\definecolor{darkred}{rgb}{0.5,0,0}
\definecolor{darkgreen}{rgb}{0,0.5,0}
\definecolor{darkblue}{rgb}{0,0,.5}

\usepackage{amsthm}
\theoremstyle{plain}

\newtheorem{theorem}{Theorem}[section]
\newtheorem{lemma}[theorem]{Lemma}
\newtheorem{proposition}[theorem]{Proposition}

\newtheorem{remark}{Remark}[section]
\numberwithin{equation}{section}

\crefname{theorem}{Theorem}{Theorems}
\crefname{lemma}{Lemma}{Lemmata}
\crefname{remark}{Remark}{Remarks}
\crefname{proposition}{Proposition}{Propositions}
\crefname{section}{Section}{Sections}
\crefname{subsection}{Subsection}{Subsections}
\crefname{equation}{}{}
\Crefname{equation}{Equation}{Equations}

 \ifcsmacro{keywords}{}{%
     \newenvironment{keywords}{\begin{quotation}\noindent \textbf{Keywords}: }{\end{quotation}}
 }

 \ifcsmacro{AMS}{}{%
     \newenvironment{AMS}{\begin{quotation} \noindent \textbf{AMS subject classifications}: }{\end{quotation}}
 }

\title{Derivative-free Bayesian Inversion Using Multiscale Dynamics}
\author{G. A. Pavliotis, A. M. Stuart, U. Vaes}

\begin{document}
\maketitle

\begin{abstract}
    Inverse problems are ubiquitous because they formalize the integration of data with mathematical models.
    In many scientific applications the forward model is expensive to evaluate,
    and adjoint computations are difficult to employ;
    in this setting derivative-free methods which involve a small number of forward model evaluations are an attractive proposition.
    Ensemble Kalman based interacting particle systems
    (and variants such as consensus based and unscented Kalman approaches)
    have proven empirically successful in this context,
    but suffer from the fact that they cannot be systematically refined to return the true solution,
    except in the setting of linear forward models~\cite{2019arXiv190308866G}.
    In this paper, we propose a new derivative-free approach to Bayesian inversion,
    which may be employed for posterior sampling or for maximum a posteriori (MAP) estimation,
    and may be systematically refined.
    The method relies on a fast/slow system of stochastic differential equations (SDEs) for the local approximation of the gradient of the log-likelihood appearing in a Langevin diffusion.
    Furthermore the method may be preconditioned by use of information from ensemble Kalman based methods (and variants),
    providing a methodology which leverages the documented advantages of those methods,
    whilst also being provably refineable.
    We define the methodology, highlighting its flexibility and many variants,
    provide a theoretical analysis of the proposed approach,
    and demonstrate its efficacy by means of numerical experiments.
\end{abstract}


\begin{AMS}
    62F15, 
    65C35, 
    65C30, 
    65N21. 
\end{AMS}

\section{Introduction}%
\label{sec:introduction}

\subsection{Overview}%
\label{sub:overview}
In this paper, we consider the inverse problem of finding an unknown parameter $\theta \in \real^d$ from data $y \in \real^K$ where
\begin{align}
    \label{eq:inverse_problem}
    y = G(\theta) + \eta,
\end{align}
with $G: \real^d \rightarrow \real^K$ a forward operator and $\eta$ the observational noise.
In the Bayesian approach to inverse problems~\cite{MR2102218,MR2652785,MR3839555},
the vectors $\theta$, $\eta$ and $y$ are treated as random variables.
If the unknown parameter and the noise are assumed to be independent and normally distributed,
with distribution parameters $\theta \sim \mathcal N(m, \Sigma)$ and $\eta \sim \mathcal N(0, \Gamma)$,
then the joint distribution of $(\theta, y)$ can be obtained from~\eqref{eq:inverse_problem}:
\[
    (\theta, y) \sim \frac{\e^{-\Phi_R(\theta; y)}}{\int_{\real^K} \!  \int_{\real^{d}}\e^{-\Phi_R(\theta; y)} \, \d \theta \, \d y}.
\]
Here $\Phi_R$ is a function that regularizes the \emph{least-squares functional} $\Phi$;
these two functions are given by
\begin{subequations}
\begin{align}
\label{eq:reg_least_squares}%
& \Phi(\theta; y) = \frac{1}{2} \abs{y - G(\theta)}[\Gamma]^2, \\
& \Phi_R (\theta; y) = \Phi(\theta; y) + \frac{1}{2} \abs{\theta - m}[\Sigma]^2,
\end{align}
\end{subequations}
with the notation,
for a positive definite matrix $A$,
\[
    \eip*{\dummy_1}{\dummy_2}[A] = \eip*{\dummy_1}{A^{-1}\dummy_2}, \qquad \abs{\dummy}[A]^2 = \eip{\dummy}{\dummy}[A],
\]
where $\eip{\dummy}{\dummy}$ is the Euclidean inner product.
By Bayes' formula,
the conditional probability density function of $\theta$ given $y$ equals
\begin{equation}
    \label{eq:bayesian_posterior}
    \pi_*(\theta) = \frac{\exp\bigl(- \Phi_R(\theta; y) \bigr)}{\int_{\real^{d}}\exp\bigl(- \Phi_R(\theta; y) \bigr) \, \d \theta} =: \frac{1}{Z(y)} \exp\bigl(-\Phi_R(\theta; y)\bigr).
\end{equation}
This probability distribution is the \emph{Bayesian posterior},
and its pointwise maximizer is the \emph{maximum a posteriori} (MAP) estimator.

There exist several approaches for solving inverse problems,
which we review in the next section.
In this paper,
we present a new derivative-free approach for~\eqref{eq:inverse_problem}.
Our method is based on a fast/slow system of SDEs,
and it may be used for sampling from the Bayesian posterior~\eqref{eq:bayesian_posterior} or
for calculating the MAP estimator.
Unlike the Ensemble Kalman Sampler (EKS),
and variants such as the unscented Kalman sampler (UKS)~\cite{unscentedki} and consensus based sampler (CBS)~\cite{cbs},
the method we present can be refined systematically in order to approach the true solution:
in the refinement limit,
it produces a stochastic process described by dynamics of the type
\begin{equation}
    \label{eq:overdamped_Langevin}
    \d \theta_t = - K \grad_{\theta} \Phi_R(\theta_t; y) \, \d t + \nu \sqrt{2 K} \, \d w_t.
\end{equation}
Here $\{w_t\}_{t\geq0}$ is a standard $d$-dimensional Brownian motion,
$K$ is a symmetric positive definite matrix,
and $\nu$ is a coefficient equal to 1 if the method is used for posterior sampling or~0 in optimization mode.
When~$\nu = 0$, equation~\eqref{eq:overdamped_Langevin} is a preconditioned gradient descent in the potential~$\Phi_R$,
and, when~$\nu = 1$, equation~\eqref{eq:overdamped_Langevin} is a preconditioned overdamped Langevin diffusion in the potential~$\Phi_R$.
From now on,
since the observation $y$ is a fixed parameter of the inverse problem,
we write~$\Phi(\theta) = \Phi(\theta; y)$ and $\Phi_R(\theta) = \Phi_R(\theta; y)$ and $Z = Z(y)$ for simplicity.


\subsection{Literature Review}%
\label{sub:literature_review}
There are two main approaches for solving inverse problems of the type~\eqref{eq:inverse_problem}: the classical approach and the Bayesian approach~\cite{MR2102218,MR2652785}.
Classical methods are generally based on an optimization problem of the form
\begin{equation}
    \label{eq:classical}
    \argmin_{\theta \in \real^d} \frac{1}{2} \eucnorm{y - G(\theta)}[X]^2 + R(\theta),
\end{equation}
for some positive definite matrix $X$ and
where $R$ is an optional regularization term.
The aim of the regularization term  is to ensure that the minimization problem is well-posed;
without this term, there may be minimizing sequences that are not bounded in $\real^d$.
See, for example, \cite{MR3774703,MR3968262,MR3104933,MR3377106,MR2746411} for a discussion of the classical approach~\eqref{eq:classical} and of regularization techniques.
A widely used form for $R(\theta)$,
known as a Tikhonov--Phillips regularization,
is given by $R(\theta) = \eucnorm{\theta - z}_{Y}^2$,
for some vector $z$ and a positive definite matrix $Y$ is .
In the classical approach,
the matrices $X$, $Y$ and the vector $z$ are generally parameters without a probabilistic interpretation.

The Bayesian approach to the inverse problem~\eqref{eq:inverse_problem},
on the other hand, relies on the statistical properties of the noise
and on the specification of a prior probability distribution which encapsulates \emph{a priori} knowledge on the unknown parameter,
as shown in~\cref{sub:overview}.
In the Bayesian framework, the optimization problem~\eqref{eq:classical} is relevant
with $X = \Gamma$ and with the Tikhonov--Phillips regularization $R(\theta) = \eucnorm{\theta - m}[\Sigma]^2$.
In this case, the solution to~\eqref{eq:classical} admits a clear interpretation:
it is the pointwise maximizer of the Bayesian posterior~\eqref{eq:bayesian_posterior},
so it can be viewed as the most likely value of the parameter given the data.
See~\cite{MR2652785} for more details on the connection between the classical and Bayesian approaches.
Often, one is interested not in a point estimator but in the statistical properties of the Bayesian posterior,
which can be used, for example, for the derivation of confidence intervals.
In most applications, the dimension of the parameter space is large,
so it is necessary to generate samples from the Bayesian posterior in order to calculate its statistical properties.

Several methods can be employed for solving the inverse problem~\eqref{eq:inverse_problem} via the optimization problem~\eqref{eq:classical}.
In a number of important applications of inverse problems,
such as parameter estimation in climate models~\cite{2020arXiv201213262D},
the derivatives of the forward operator $G$ are unavailable or too computationally expensive to obtain,
so, in this literature review,
we only briefly review gradient-based methods and we focus mostly on derivative-free methods.

When the derivatives of the forward operator $G$ are available,
a natural approach is to employ the gradient descent algorithm or one of its variants:
we mention, for example, the conjugate gradient descent~\cite{MR0060307}, the stochastic gradient descent~\cite{MR42668,MR1993642},
the Barzilai--Borwein method~\cite{MR967848},
and other gradient-based optimization techniques that rely on interacting particle systems~\cite{NIPS2015_d18f655c,2020arXiv200707704B,borovykh2020interact}.
One may also recur to the Newton or Gauss--Newton methods and their variants~\cite{MR2299676,akcelik2002parallel},
or to methodologies based on the Levenberg--Marquadt method~\cite{MR10666,MR153071,MR1435869}.

When the derivatives of the forward operator $G$ are unavailable, on the other hand,
derivative-free methods are required for solving the optimization problem~\eqref{eq:classical}.
A comprehensive presentation of standard derivative-free optimization methods is given in~\cite{MR2487816},
and a review focusing on recent trends and developments in this field is given in~\cite{MR3963507}.
A simple approach,
which was popular in many of the early works on derivative-free optimization,
is to employ a finite difference gradient approximation in place of the exact gradient in a derivative-based method;
see, for example,~\cite{MR303723,MR697184}.
Many other general-purpose derivative-free optimization methods can be employed for solving~\eqref{eq:classical},
and we mention, for example,
simulated annealing~\cite{MR702485},
particle-swarm optimization~\cite{kennedy1995particle},
and consensus-based optimization (CBO)~\cite{MR3597012,MR3804923}.
One may also use methods based on the ensemble Kalman filter,
which rely on the quadratic structure of the loss function in~\eqref{eq:classical}.
Ensemble Kalman methods, introduced for dynamical state estimation in~\cite{MR2555209},
were extended as derivative-free Bayesian inverse problem solvers
in~\cite{Chen2012,MR3036193}; they were modified to become derivative-free optimizers in~\cite{MR3041539},
a method we refer to as Ensemble Kalman Inversion (EKI).
The EKI has been shown,
both theoretically in simple settings and empirically,
to perform very well in the context of inverse problems~\cite{MR3041539,MR3654885,MR3764752},
and it has also been applied successfully for training neural networks~\cite{MR3998631}.
A recent variant on the EKI, unscented Kalman inversion (UKI),
shows significant promise for problems in which the parameter dimension is low,
but the forward model is expensive to evaluate and hard to differentiate \cite{unscentedki}.
Another alternative is the method developed in~\cite{2018arXiv180508034H},
which is based on similar ideas for gradient approximation but aims to drive a single distinguished particle to the optimizer.

Likewise, there exist several methods for solving~\eqref{eq:inverse_problem} via the Bayesian approach,
i.e.\ for generating samples from the Bayesian posterior~\eqref{eq:bayesian_posterior}.
If the derivatives of the log-posterior are available,
the simplest option is to rely on a Langevin diffusion of the type~\eqref{eq:overdamped_Langevin},
which enjoys the property of transforming any initial distribution into the Bayesian posterior~\eqref{eq:bayesian_posterior} in the longtime limit $t \to \infty$.
One may also employ higher-dimensional stochastic dynamics that admit the Bayesian posterior as a marginal of their ergodic measure,
such as the underdamped Langevin dynamics~\cite{MR3509213,pavliotis2011applied} or the generalized Langevin dynamics~\cite{MR2793823,MR3986068,GPGSUV_2020}.

Another standard and related approach for sampling from a high-dimensional probability density is to use a Markov chain Monte Carlo method (MCMC),
i.e.\ to construct a Markov chain whose unique invariant distribution is the target density.
To this end, the most widely used method is the Metropolis--Hastings algorithm~\cite{metropolis1953equation,MR3363437} (MH).
All that is required to define a MH algorithm is a proposal distribution,
which may or may not be based on the derivatives of the target density (or of its logarithm).
We mention, for example, the Metropolis-adjusted Langevin dynamics (MALA, which uses the derivative) and the random walk MH method (RWMH, which does not).
There is also a substantial literature on the computation of,
or exploiting, Gaussian approximations of the posterior;
see~\cite{MR3233941} and the references therein.
There is extensive literature on the convergence properties and optimal parametrization of these methods,
and on their connections with overdamped Langevin diffusions~\cite{MR1440273,MR1428751,MR1888450} in the high-dimensional limit~\cite{MR1440273,MR1428751,MR1888450}.
See also~\cite{MR3349007} for a study of the connection of high-dimensional RWMH with overdamped Langevin dynamics in the transient regime,
and~\cite{MR2583309} for a proof of convergence of MALA to an overdamped Langevin diffusion in the small timestep limit in fixed dimension.

In recent years, there has also been significant activity devoted to developing sampling methods based on interacting particle systems,
which can leverage recent advances in parallel computing.
These include, for example,
sequential Monte Carlo samplers~\cite{del2006sequential},
interacting particle MCMC methodologies~\cite{MR3060209} and~\cite{MR3747563}, Stein variational gradient descent~\cite{lu2018scaling,Liu2016},
the ensemble Kalman sampler (EKS)~\cite{2019arXiv190308866G},
and affine-invariant Langevin dynamics (ALDI)~\cite{garbuno2020affine}.
The derivative-free formulations of the latter two methods were proposed specifically for Bayesian inverse problems
-- they rely on the least-squares structure~\eqref{eq:reg_least_squares} of the log-posterior --
and they were shown to produce good approximate samples of the posterior distribution at a relatively low computational cost.
Both EKS and ALDI are strongly related to~\eqref{eq:overdamped_Langevin}:
in the linear setting, they are based on a system of preconditioned overdamped Langevin diffusions,
with a time-dependent preconditioner given by the covariance of the ensemble.
ALDI improves upon EKS by incorporating a correction term which guarantees that
the ergodic measure of the finite-dimensional particle system is the product measure of $J$ copies of the target distribution,
where $J$ denotes the number of particles.
A variant on the EKS using the unscented transform (UKS) has recently been proposed \cite{unscentedki},
and a generalization of CBO to sampling (CBS) has recently been proposed~\cite{cbs};
both UKS and CBS are derivative-free.

The parallel MCMC method of~\cite{MR3747563},
as well as the ensemble Kalman based methods for Bayesian inverse problems,
i.e.\ EKI, EKS and ALDI,
enjoy the property of being affine invariant in the sense of~\cite{MR2600822};
see also~\cite{greengard} and~\cite{MR3362507}.
As the terminology indicates,
affine invariant methods are insensitive to affine transformations of the regularized least-squares functional~$\Phi_R$,
which makes them particularly well-suited in cases where~$\Phi_R$ exhibits strong anisotropy at its minimizer.
The affine invariance of EKS, ALDI and the ensemble Kalman--Bucy filter was demonstrated carefully in~\cite{garbuno2020affine},
where the authors also show that the Bayesian posterior is invariant and ergodic under ALDI.
Around the same time, it was observed that
the rate of convergence to equilibrium for the nonlocal PDEs associated with EKI and EKS
was independent of the parameters of the regularized least-squares functional $\Phi_R$,
in the simple case of a linear forward model~\cite{2019arXiv190308866G,2019arXiv191007555C};
this independence is in fact a consequence of affine invariance,
although this fact is not identified in these references.

As mentioned in~\cref{sub:overview},
the method we present in this paper is based on a fast/slow system of stochastic differential equations,
and it may be used both for MAP estimation and posterior sampling.
Multiscale methods have been used before for optimization purposes.
A multiscale dynamics is employed~\cite{MR3819749} for smoothing the loss function associated with deep neural networks,
and the method is revisited later in~\cite{2019arXiv190504121K}.
A similar multiscale dynamics is also employed in~\cite{quer2018importance} for calculating convolutions,
with the aim of reducing metastability in the context of molecular dynamics.
See also~\cite{MR1402204} for information on how smoothing the objective function by convolution with a Gaussian kernel
can be helpful in optimization schemes.
The method we propose in this paper is based on similar ideas,
in that it employs a fast/slow system SDEs for approximating the gradient of the loss function,
but it is gradient-free and relies on a different multiscale system.
In addition, we demonstrate how preconditioning can be incorporated in the method
in order to approach the solution to~\eqref{eq:overdamped_Langevin} with an appropriate symmetric, positive definite matrix~$K$.
We also show how a good preconditioner~$K$ can be constructed using information from ALDI;
similar approaches can be used based on information obtained through
EKS, UKS or~CBS.

Our work is aimed at sampling posterior distributions which are not Gaussian.
The paper \cite{MR3400030} demonstrates clearly that
standard ensemble Kalman based methods for inverse problems do not reproduce the correct posterior distribution in the large particle size limit in the non-Gaussian setting.
Although there is interesting empirical work which addresses this shortcoming through iteration of ensemble Kalman filters~\cite{sakov2012iterative,bocquet2014iterative},
it is not clear that this methodology may be applied systematically to arbitrary inverse problems.
Our proposed methodology, on the other hand,
addresses shortcomings of standard ensemble methods in this setting and is founded on refineable approximations which,
in certain limits, will reproduce the true posterior distribution.

Although our method applies, in principle, to multimodal distributions,
it will not be efficient in this scenario:
this is because our method is based on approximation of an overdamped
Langevin equation, and hence may suffer from metastability issues when
the posterior is multi-modal.
For this reason our focus is on non-Gaussian unimodal distributions;
our primary purpose in this paper is not to address multi-modality.
Indeed it is likely that a generic solution to the problem of multi-modality will be very hard to find~\cite{woodard2009sufficient}
-- problem-specific multi-modal approaches are more likely to yield fruitful research.
There is, however, interesting work by~\cite{LMS16}
which uses a localized sample covariance matrix in a parallel MCMC method to address this problem in a generic fashion;
we briefly touch on this in~\cref{sub:accelerating_convergence_using_preconditioning}.
The idea of using localized covariances is incorporated into ALDI in~\cite{MR4248267} in order to extend the range of applicability of the method beyond the unimodal setting.

\subsection{Our Contributions}%
\label{sub:our_contributions}

This paper, then, is focused on the construction of provably refineable derivative-free methods for Bayesian inverse problems
characterized by unimodal but non-Guassian posterior distributions.
The primary contributions in this paper are the following:
\begin{itemize}
    \item
        We present a novel method based on a multiscale dynamics for MAP estimation and posterior sampling in Bayesian inverse problems.
        We discuss possible variations of the method and present a fully practical numerical discretization.

    \item
        In addition to motivating the method with formal arguments,
        we prove the pathwise convergence of the solution it produces to a gradient descent or to an overdamped Langevin diffusion,
        depending on whether the method is used for optimization or sampling, respectively.
        We also obtain a strong convergence estimate for the numerical discretization of the multiscale dynamics.

    \item
        We present numerical experiments demonstrating the efficiency of the method,
        for the purposes of both sampling and optimization.
        We consider first a standard low-dimensional test problem and then a high-dimensional inverse problem where the forward model requires the solution of an elliptic~PDE.

    \item
        We show how a significant improvement in performance can be obtained by preconditioning the method using information from ALDI.
\end{itemize}
The rest of the paper is organized as follows.
In \cref{sec:presentation_of_the_method},
we introduce the multiscale method and present our main results.
In \cref{sec:numerical_experiments}, we present numerical experiments demonstrating the efficacy of the method,
both for low-dimensional and high-dimensional parameter spaces,
and we show how preconditioning can be incorporated in the method.
In \Cref{sec:proof_of_main_results}, we prove our main convergence results.
\Cref{sec:conculsions} is reserved for conclusions and perspectives for future work.

\section{Presentation of the Method and Main Results}%
\label{sec:presentation_of_the_method}
This section is organized as follows:
in \cref{sub:continuous_time_dynamics},
the multiscale method is presented as a continuous-time dynamics and
motivated by formal arguments.
In \cref{sub:numerical_discretization}, a fully practical time discretization of the continuous dynamics is presented.
\Cref{sub:main_results} then presents the statements of our main results,
the proofs of which are given in~\cref{sec:proof_of_main_results}.

\subsection{Continuous-time Dynamics}%
\label{sub:continuous_time_dynamics}
Our method is based on a multiscale system of stochastic differential equations (SDEs):
A slow variable is employed for the purposes of finding the MAP estimator or sampling from the Bayesian posterior,
and several fast variables provide information on the variation of the least-squares functional in the vicinity of the slow variable.
At any time, the drift for the slow variable is calculated based on the values of the forward functional at the positions of these fast explorers,
using a projected gradient approximation similar in structure to that used in the methods for inversion and sampling based on the ensemble Kalman filter (EnKF);
see~\cite{MR3654885} (in the context of optimization) and \cite{2019arXiv190308866G} (in the context of sampling).
The idea of employing an approximation based on ensemble Kalman methods for specifying the drift of a single distinguished particle
is inspired by the paper~\cite{2018arXiv180508034H},
in which the authors proposed to use stochastic differences as a surrogate for gradients in an ensemble based optimization context.

In most applications, the space of the unknown parameter is $\real^d$ but,
for simplicity of the analysis presented in~\cref{sec:proof_of_main_results},
we also consider the case where this space is the $d$-dimensional torus $\torus^d$.
Therefore, we denote the parameter space by $\statespace^d$, with $\statespace = \real$ or $\statespace = \torus$.
At the continuous-time level, our method is based on the following system of interacting SDEs:
\begin{subequations}
\label{eq:multiscale:evolution_theta}%
\begin{align}
    \label{eq:multiscale:evolution_theta_theta}
    \dot {\theta}
    &= - \frac{1}{J \sigma^2} \sum_{j=1}^{J} \eip*{G(\theta^{(j)}) - G(\theta)}{G(\theta) - y}[\Gamma] \, (\theta^{(j)} - \theta)
    -  C(\Xi) \Sigma^{-1} (\theta - m) + \nu \sqrt{2 C(\Xi)} \, \dot {w}, \\
    \theta^{(j)} &= \theta + \sigma \,  \xi^{(j)}, \qquad \qquad \qquad \qquad \qquad \qquad \qquad \qquad \qquad \qquad \qquad \quad
      j = 1, \dotsc, J, \\
    \dot {\xi}^{(j)} &= - \frac{1}{\delta^2} \, \xi^{(j)} + \sqrt{\frac{2}{\delta^2}} \, \dot {w}^{(j)},
    \qquad \xi^{(j)}(0) \sim \mathcal N(0, {I_d}),  \qquad \qquad \qquad \qquad
      j = 1, \dotsc, J,
\end{align}
\end{subequations}
where $\theta \in \statespace^d$, $\Xi = (\xi^{(1)}, \dotsc, \xi^{(J)}) \in (\real^d)^J$,
$I_d$ is the  $\real^{d \times d}$ is the identity matrix,
the processes $w$ and $\{w^{(j)}\}_{j=1}^{J}$ are independent standard Brownian motions and
\begin{align*}
 &C(\Xi) = \frac{1}{J} \sum_{j=1}^{J} (\xi^{(j)} \otimes \xi^{(j)}).
\end{align*}
The processes $\{\xi^{(j)}\}_{j = 1}^{J}$ are stationary Ornstein--Uhlenbeck (OU) processes
with invariant measure~$\mathcal N(0, {I_d})$ and autocorrelation function $\e^{-\abs{t}/\delta^2} {I_d}$.
The parameter $\delta$ can therefore be viewed as the square root of the characteristic time scale of the fast processes.
The coefficient $\nu \in \{0, 1\}$ controls whether noise should be included in the equation for $\theta$:
if $\nu = 0$, then~\eqref{eq:multiscale:evolution_theta} is a method for finding the minimizer of the regularized least-squares functional $\Phi_R$,
i.e.\ the MAP estimator;
if $\nu = 1$, then~\eqref{eq:multiscale:evolution_theta} is a method for sampling from the posterior distribution $\frac{1}{Z}\e^{-\Phi_R}$,
where $Z$ is the normalization constant given in~\eqref{eq:bayesian_posterior}.

Note that~\eqref{eq:multiscale:evolution_theta} can also be employed without the prior regularization and with $\nu = 0$,
as a method for finding the minimizer of the non-regularized least-squares functional $\Phi$.
This formally corresponds to taking the prior covariance to be infinite, i.e.\ $\Sigma = \infty {I_d}$,
and can be useful when prior knowledge of the unknown parameter $\theta$ is not available,
or is not needed because the problem is over-determined.
In this case, any parameter $\theta \in G^{-1}(y)$ is a steady state of~\eqref{eq:multiscale:evolution_theta},
which is not the case for the alternative derivative-free formulation~\eqref{eq:alternative} we present below.

\begin{remark}[Connection with the stochastic gradient descent]
    The method~\eqref{eq:multiscale:evolution_theta} is most useful with $\sigma$ small,
    in which case,
    neglecting quadratic or smaller terms in $\sigma$,
    the following approximation holds:
    \begin{equation}
        \label{eq:linear_approximation}
        G(\theta^{(k)}) - G(\theta) \approx (\theta^{(k)} - \theta) \cdot \grad G(\theta).
    \end{equation}
    Using this approximation,
    we can rewrite the equation for $\theta$ in~\eqref{eq:multiscale:evolution_theta} as
    \begin{align*}
        \dot {\theta} &\approx - \frac{1}{J} \sum_{k=1}^{J} \left( \xi^{(k)} \otimes \xi^{(k)} \right) \, \grad \Phi(\theta)
        - C(\Xi) \Sigma^{-1} (\theta - m) + \nu \sqrt{2 C(\Xi)} \, \dot {w} \\
                    &= - C(\Xi) \, \grad \Phi_R(\theta) + \nu \sqrt{2 C(\Xi)} \, \dot {w}.
    \end{align*}
    The term $C(\Xi) \, \grad \Phi_R(\theta)$ can be viewed as a projection of $\grad \Phi_R(\theta; y)$ on the subspace spanned by
    $\{ \xi^{(1)}, \dotsc, \xi^{(J)} \}$,
    which shows a link with the stochastic gradient descent algorithm.
    For large $J$, it holds formally that $C(\Xi) \approx {I_d}$ at all times,
    so the equation for $\theta$ reduces to a gradient descent when $\nu = 0$,
    or to the overdamped Langevin equation if $\nu = 1$,
    both both with respect to the potential $\Phi_R$.
\end{remark}

\begin{remark}
    Note that~\eqref{eq:linear_approximation} holds exactly when $G$ is linear,
    for all $\sigma > 0$.
    In this case and in the presence of noise (i.e.\ when $\nu = 1$), equation~\eqref{eq:multiscale:evolution_theta} admits as invariant measure the distribution
    \begin{align}
        \label{eq:invariant_measure}
        \rho_{\infty}(\theta, \Xi) = \frac{1}{Z} \, \exp \left( -\Phi_R(\theta) \right) \, g(\xi^{(1)}; 0, {I_d}) \dotsc \, g(\xi^{(J)}; 0, {I_d}).
    \end{align}
    The associated marginal distribution for $\theta$ is given by $\frac{1}{Z}\e^{-\Phi_R}$,
    which is precisely the Bayesian posterior distribution.
    To show that~\eqref{eq:invariant_measure} is indeed the unique invariant distribution when $G$ is linear,
    we note that the Fokker--Planck operator associated with~\eqref{eq:multiscale:evolution_theta} in this case is given by~\cite[Chapter 4]{pavliotis2011applied}
    \begin{align*}
        \mathcal L^{\dagger} \rho
        &= \grad_{\theta} \cdot \Bigl(  C(\Xi) \bigl( \grad \Phi_R(\theta) \rho + \grad_{\theta} \rho  \bigr) \Bigr)
        + \frac{1}{\delta^2} \sum_{j=1}^{J} \grad_{\xi^{(j)}} \cdot \left( \xi^{(j)} \rho + \grad_{\xi^{(j)}} \rho\right) \\
        &=  \grad_{\theta} \cdot \left(  \rho_{\infty} C(\Xi)  \grad_{\theta} \left(\frac{\rho}{\rho_{\infty}} \right)  \right)
        + \frac{1}{\delta^2} \sum_{j=1}^{J} \grad_{\xi^{(j)}} \cdot \left( \rho_{\infty} \grad_{\xi^{(j)}} \left( \frac{\rho}{\rho_{\infty}}\right)\right)
        = \grad \cdot \left( \rho \mathcal D \grad \log \left( \frac{\rho}{\rho_{\infty}}\right) \right),
    \end{align*}
    where $\mathcal D$ is the $\real^{d(J + 1) \times d(J + 1)}$ block diagonal matrix with diagonal blocks $C(\Xi), \frac{1}{\delta^2} I_d, \dotsc, \frac{1}{\delta^2} I_d$.
    It is clear that $\rho_{\infty}$ in~\eqref{eq:invariant_measure} is in the kernel of this operator.
    To show formally that the invariant measure is unique,
    it suffices to multiply both sides of the equation $\mathcal L^\dagger \rho = 0$ by $\rho/\rho_{\infty}$ and to integrate over the state space $\statespace^d \times (\real^d)^J$,
    which gives
    \begin{align*}
        \int_{\statespace^d} \int_{\real^d} \dots \int_{\real^{d}} \left( C(\Xi) \eucnorm{\grad_{\theta} \left(\frac{\rho}{\rho_{\infty}} \right)}^2
        + \frac{1}{\delta^2}\sum_{j=1}^{J} \eucnorm{\grad_{\xi^{(j)}}  \left(\frac{\rho}{\rho_{\infty}}\right)}^2 \right) \, \rho_{\infty}(\theta, \Xi) \, \d \xi^{(1)} \dots \d \xi^{(J)} \, \d \theta = 0,
    \end{align*}
    and therefore $\rho = \rho_{\infty}$ necessarily.
\end{remark}


For the purposes of analysis,
we also consider a simplified version of~\eqref{eq:multiscale:evolution_theta}
in which the coefficient of the noise is independent of the fast processes $\xi^{(1)}, \dotsc, \xi^{(J)}$:
\begin{subequations}
    \label{eq:alternative_derivative-free}
    \begin{align}
        \label{eq:alternative_theta}%
        \dot {\theta} &= - \frac{1}{J \sigma^2} \sum_{j=1}^{J} \eip*{G(\theta^{(j)}) - G(\theta)}{G(\theta) - y}[\Gamma] \, (\theta^{(j)} - \theta)
        - C(\Xi) \Sigma^{-1} (\theta - m) + \nu \sqrt{2} \dot {w},\\
        \label{eq:alternative_theta_other}%
    \theta^{(j)} &= \theta + \sigma \, \xi^{(j)}, \qquad j = 1, \dotsc, J.\\
        \label{eq:alternative_theta_xi}%
    \dot {\xi}^{(j)} &= - \frac{1}{\delta^2} \, \xi^{(j)} + \sqrt{\frac{2}{\delta^2}} \, \dot {w}^{(j)}.
    \end{align}
\end{subequations}
Equation~\eqref{eq:alternative_theta} admits the same formal limit as $J \to \infty$ or $\delta \to 0$ as~\eqref{eq:multiscale:evolution_theta_theta},
but it is simpler to analyze because the noise is additive.
We note, however,
that the invariant measure associated to~\cref{eq:alternative_theta,eq:alternative_theta_other,eq:alternative_theta_xi} for finite $J$ differs from~\eqref{eq:invariant_measure},
even in the case of a linear forward model.

\begin{remark}
    [Second alternative derivative-free formulation]
    \label{rem:etic}
    Instead of \eqref{eq:multiscale:evolution_theta} or~\eqref{eq:alternative_derivative-free},
    we could also use a system of equations of the form
\begin{subequations}
\label{eq:alternative}%
\begin{align}
    \label{eq:alternative_a}%
    \dot {\theta} &= - \frac{1}{J \sigma^2} \sum_{j=1}^{J} \, \left(\Phi_R(\theta^{(j)}) - \Phi_R(\theta)\right)
    \left(\theta^{(j)} - \theta\right) + \nu \sqrt{2 C(\Xi)} \, \dot {w}, \\
    \label{eq:alternative_b}%
    \theta^{(j)} &= \theta + \sigma \, \xi^{(j)}, \qquad j = 1, \dotsc, J,\\
    \label{eq:alternative_c}%
    \dot {\xi}^{(j)} &= - \frac{1}{\delta^2} \, \xi^{(j)} + \sqrt{\frac{2}{\delta^2}} \, \dot {w}^{(j)}.
\end{align}
\end{subequations}
This formulation has two advantages over~\cref{eq:multiscale:evolution_theta,eq:alternative_derivative-free}:
it does not require the prior distribution to be Gaussian,
and it does not rely on the specific quadratic structure of $\Phi_R$ in~\eqref{eq:reg_least_squares},
making it more generally applicable.
However, like~\eqref{eq:alternative_derivative-free},
the system of equations~\eqref{eq:alternative} does not admit~\eqref{eq:invariant_measure} as invariant measure when~$J$ is finite,
not even in the case of a linear forward model.
\end{remark}
\newcommand{\approxtheta}[1]{\hat {\theta}_{#1}}
\newcommand{\approxthetatilde}[1]{\tilde {\theta}_{#1}}
\newcommand{\approxXi}[1]{\hat {\Xi}_{#1}}
\newcommand{\approxxi}[2]{\hat {\xi}_{#1}^{(#2)}}
\newcommand{\approxupsilon}[2]{\hat {\upsilon}_{#1}^{(#2)}}
\newcommand{\approxUpsilon}[1]{\hat {\Upsilon}_{#1}}
\newcommand{\approxvartheta}[1]{\hat {\vartheta}_{#1}}

\subsection{Numerical Discretization}%
\label{sub:numerical_discretization}

To integrate~\eqref{eq:multiscale:evolution_theta} numerically,
we employ the Euler--Maruyama method for $\theta$ and a closed formula for the exact (in law) solution of the OU process for $\{\xi^{(j)}\}_{j=1}^J$.
We use the notation~$\Delta$ to denote the time step and the notation $(\approxtheta{n}, \approxXi{n})$,
with $\approxXi{n} = (\approxxi{n}{1}, \dotsc, \approxxi{n}{J})$,
to denote the numerical approximation of $(\theta_{n \Delta}, \Xi_{n \Delta})$,
i.e.\ the numerical approximation of the continuous-time solution at time $n \Delta$.
We also denote by $N$ the total number of iterations and by $T = N \Delta$ the final time of the simulation.
The numerical scheme we propose reads
\begin{subequations}
\label{eq:multiscale:evolution_theta_discrete}%
\begin{align}
    \label{eq:multiscale:evolution_theta_discrete_theta}%
    \approxtheta{n+1}
    &= \approxtheta{n}
    - \frac{1}{J \sigma} \sum_{j=1}^{J} \eip*{G(\approxtheta{n} + \sigma \approxxi{n}{j}) - G(\approxtheta{n})}{G(\approxtheta{n}) - y}[\Gamma] \, \approxxi{n}{j} \, \Delta \\
    \notag
    &\qquad -  C(\approxXi{n}) \Sigma^{-1} (\approxtheta{n} - m) \Delta + \nu \sqrt{2 D(\approxXi{n}) \Delta} \, x_n, \\
    \label{eq:multiscale:evolution_theta_discrete_xis}%
    \approxxi{n+1}{j} &= \e^{-\frac{\Delta}{\delta^2}}\approxxi{n}{j} + \sqrt{1 - \e^{-\frac{2\Delta}{\delta^2}}} x_n^{(j)}, \qquad \approxxi{0}{j} \sim \mathcal N(0, I_d), \qquad j = 1, \dotsc, J,
\end{align}
\end{subequations}
where $x_n = \Delta^{-1/2}(w_{(n+1)\Delta} - w_{n \Delta}) \sim \mathcal N(0, I_d)$ and $x_n^{(j)}$, for $n = 0, \dotsc, N-1$ and $j = 1, \dotsc, J$,
are independent $\mathcal N(0, {I_d})$ random variables.
Here $D(\approxXi{n}) = C(\approxXi{n})$ or $D(\approxXi{n}) = {I_d}$
depending on whether a solution to~\eqref{eq:multiscale:evolution_theta} or to~\eqref{eq:alternative_derivative-free} is sought,
respectively.
At the numerical level, we can consider the limit $\delta \to 0^+$,
in which case the numerical scheme~\eqref{eq:multiscale:evolution_theta_discrete} simplifies to
\begin{equation}
\label{eq:multiscale:evolution_theta_discrete_simplified}%
\begin{aligned}
    \approxtheta{n+1}
    &= \approxtheta{n}
    - \frac{1}{J \sigma} \sum_{j=1}^{J} \eip*{G(\approxtheta{n} + \sigma \approxxi{n}{j}) - G(\approxtheta{n})}{G(\approxtheta{n}) - y}[\Gamma] \, \approxxi{n}{j} \, \Delta \\
    &\qquad -  C(\approxXi{n}) \Sigma^{-1} (\approxtheta{n} - m) \Delta + \nu \sqrt{2 D(\approxXi{n}) \Delta} \, x_n,
\end{aligned}
\end{equation}
where $\approxxi{n}{j}$, for $n = 0, \dotsc, N-1$ and $j = 1, \dotsc, J$, are drawn independently from $\mathcal N(0, {I_d})$.
In the noise-free case $\nu = 0$,
the algorithm in form~\eqref{eq:multiscale:evolution_theta_discrete_simplified} is precisely
what was proposed and implemented in the paper~\cite{2018arXiv180508034H},
in the context of optimizing parameters of neural networks, and inspired the work presented here.
This algorithm is simpler to implement and analyze than~\eqref{eq:multiscale:evolution_theta_discrete},
but may not always be the best option for the purposes of sampling and optimization.
Although we have found small values of $\delta$ to be preferable in our numerical experiments,
it is indeed conceivable that problems with rugged energy landscapes may benefit from
stronger correlation between successive descent directions.
We leave this question for future work and,
in all the numerical experiments presented in~\cref{sec:numerical_experiments},
we use the scheme~\eqref{eq:multiscale:evolution_theta_discrete} with~$\delta > 0$ (but in most examples very small).
We do, however, analyze the convergence of~\eqref{eq:multiscale:evolution_theta_discrete_simplified} theoretically in \cref{sec:proof_of_main_results},
as a first step towards proving the convergence of~\eqref{eq:multiscale:evolution_theta_discrete}.

\subsection{Main Results}%
\label{sub:main_results}
In this section, we present and comment on our main results, which are proved in \cref{sec:proof_of_main_results}.
In order to simplify the analysis,
we assume throughout that the state space of the parameter is the $d$-dimensional torus $\torus^d$, rather than $\real^d$,
and that the prior distribution is the uniform density over~$\torus^d$,
in which case $\Phi_R = \Phi$.
In all our results, the minimum regularity requirement on the forward model is that $G \in C^{2}(\torus^d, \real^K)$,
but we also present refined estimates in the case of a more regular forward model, namely if $G \in C^{3}(\torus^d, \real^K)$.
In particular, the case of a linear forward model is excluded.

We first present a strong convergence result for the dynamics~\eqref{eq:alternative_derivative-free},
which coincides with the dynamics~\eqref{eq:multiscale:evolution_theta} in the absence of noise.
More precisely, \cref{thm:convergence_continuous_equations} establishes, in the joint limit~$\delta \to 0$ and $\sigma \to 0$,
the pathwise convergence of the stochastic process $\{\theta_t\}_{t \in [0, T]}$
to the solution $\{\vartheta_t\}_{t \in [0, T]}$ of the averaged equation
\begin{equation}
    \label{eq:averaged_equation}
    \dot {\vartheta_t} = - \grad \Phi_R(\vartheta_t) + \nu \sqrt{2} \, \dot w_t, \qquad \vartheta_0 = \theta_0.
\end{equation}
Note that the Brownian motion and initial condition in this equation are the same as in~\eqref{eq:alternative_derivative-free},
which allows to establish a strong convergence estimate.
\begin{theorem}
    \label{thm:convergence_continuous_equations}
    Assume that $\{\theta_t, \xi^{(1)}_t, \dotsc, \xi^{(J)}_t\}_{t \in [0, T]}$ is a solution to~\eqref{eq:alternative_derivative-free}
    supplemented with any initial condition $(\theta_0, \xi^{(1)}_0, \dotsc, \xi^{(J)}_0)$ such that
    \begin{equation}
        \label{eq:initial_condition_xi}
        (\xi^{(1)}_0, \dotsc, \xi^{(J)}_0) \sim \mathcal N(0, I_d) \times \dotsb \times \mathcal N(0, I_d).
    \end{equation}
    If $G \in C^{2}(\torus^d, \real^K)$,
    then for any $p > 1$, any $J > 0$ and any $T > 0$ there is $C = C(p, T, J)$ such that
    \begin{equation}
        \label{eq:convergence_continuous}
        \forall (\delta, \sigma) \in \real_+ \times \real_+, \qquad
        \expect \left( \sup_{0 \leq t \leq T} \abs{\theta_t - \vartheta_t}^p \right) \leq C (\delta^p + \sigma^{\beta p}).
    \end{equation}
    The exponent $\beta$ is defined as follows:
    \begin{equation}
        \label{eq:definition_beta}
        \beta =
        \begin{cases}
            1 &\text{if } G \in C^{2}(\torus^d, \real^K), \\
            2 &\text{if } G \in C^{3}(\torus^d, \real^K).
        \end{cases}
    \end{equation}
\end{theorem}

\begin{remark}
    It should in principle be possible, but may prove technically challenging,
    to extend our findings to $\real^d$, Gaussian priors and unbounded vector fields by using results from~\cite{pardoux2001poisson}.
    Furthermore, although the convergence result is over a finite time interval,
    it is to be expected that convergence of invariant measures might also be established,
    using (for example) the ideas in \cite{MattinglyStuartTretyakov2010} or \cite{bally1996law,bally1996lawb}.
    In this regard we notice that we have focused on pathwise convergence on finite time intervals;
    typically only weak convergence results would be required to obtain convergence of the invariant measure.
    Indeed weak convergence results will be needed to study the formulation from~\cref{rem:etic} because the noise is multiplicative.
\end{remark}

\begin{remark}
    Since knowing the convergence rates with respect to $\delta$ and $\sigma$ is helpful for the parametrization of the method in practice,
    we opted to explicitly consider both cases.
    The critical change from two to three derivatives occurs because
    three or more derivatives are required to exploit the mean zero property of third moments of $\Xi$.
    One might wonder whether an even higher regularity of $G$ could lead to better convergence rates in the limit $\sigma \to 0$.
    An inspection of the proof of~\cref{thm:convergence_continuous_equations} reveals this is not the case.
\end{remark}
In order to balance the two error terms on the right-hand side of~\eqref{eq:convergence_continuous},
one may choose $\sigma \propto \delta$ when $G \in C^{2}(\torus^d, \real^K)$,
or just $\sigma \propto \sqrt{\delta}$ when $G \in C^{3}(\torus^d, \real^K)$.
Since a larger value of $\sigma$ seems to favor exploration of the state space,
as suggested by the numerical experiments in \cref{sub:low_dimensional_parameter_space},
choosing $\sigma \propto \sqrt{\delta}$ might indeed be advantageous for convergence when $G \in C^{3}(\torus^d, \real^K)$.

Next, we present the counterpart of \cref{thm:convergence_continuous_equations} for the numerical discretizations~\eqref{eq:multiscale:evolution_theta_discrete}.
We note that a weaker metric is employed in this result than in~\eqref{eq:convergence_continuous}.
\begin{theorem}
    \label{thm:convergence_discrete_equations_full}
    Assume that $\{\approxtheta{n}, \approxxi{n}{1}, \dotsc, \approxxi{n}{J}\}_{n=0}^{N}$ is a solution to~\eqref{eq:multiscale:evolution_theta_discrete} with $D(\dummy) = I_d$
    and an initial condition $(\theta_0, \xi^{(1)}_0, \dotsc, \xi^{(J)}_0)$ satisfying~\eqref{eq:initial_condition_xi}.
    If $G \in C^2(\torus^d, \real^K)$,
    then for any $J > 0$ and any~$T > 0$,
    there exists a constant $C = C(T, J)$ such that
    \begin{equation}
        \label{eq:general_discrete_bound}
        \forall  (\delta, \sigma, \Delta) \in \real_+^3, \qquad
        \sup_{0 \leq n \leq T/\Delta} \expect \abs{\approxtheta{n} - \vartheta_{n \Delta}}^2
        \leq C \left( \Delta + \sigma^{2\beta} + \log(1 + \delta^{-1}) \, \delta^2 \right).
    \end{equation}
    The exponent $\beta$ is defined as in~\eqref{eq:definition_beta}.
\end{theorem}

\begin{remark}
It is straightforward to obtain an error bound for~\eqref{eq:multiscale:evolution_theta_discrete_simplified} from this result.
Indeed, denoting by~$\approxthetatilde{n}$ the solution to~\eqref{eq:multiscale:evolution_theta_discrete_simplified} in order to differentiate it from the solution to~\eqref{eq:multiscale:evolution_theta_discrete},
it holds by the triangle inequality and~\eqref{eq:general_discrete_bound} that, for all $(\delta, \sigma, \Delta) \in \real_+^3$ and all $\varepsilon > 0$,
\begin{align*}
        \sup_{0 \leq n \leq N} \expect \abs{\approxthetatilde{n} - \vartheta_{n \Delta}}^2
        \leq \, &\left( 1 + \frac{1}{\varepsilon} \right) \sup_{0 \leq n \leq N} \expect \abs{\approxthetatilde{n} - \approxtheta{n}}^2  \\
        &+ (1 + \varepsilon) \,  C \left( \Delta + \sigma^{2\beta} + \log(1 + \delta^{-1}) \, \delta^2 \right).
\end{align*}
Here we used that, by Young's inequality, it holds $(a + b)^2 \leq (1 + \varepsilon) a^2 + \left(1 + \frac{1}{\varepsilon}\right)b^2$ for any $\varepsilon > 0$ and any~$a, b \in \real$.
It is clear that, if $\approxxi{0}{j}, \approxxi{1}{j}, \approxxi{2}{j}, \dotsc$ in~\eqref{eq:multiscale:evolution_theta_discrete_simplified} coincide with $\approxxi{0}{j}, x_0^{(j)}, x_1^{(j)}, \dotsc$ in~\eqref{eq:multiscale:evolution_theta_discrete} for all $1 \leq j \leq J$,
then the first term on the right-hand side vanishes in this limit.
Therefore, letting $\delta \to 0$ and then $\varepsilon \to 0$, we deduce
\begin{equation}
    \label{eq:error_bound_simplified_discrete}
    \forall (\sigma, \Delta) \in \real_+ \times \real_+, \qquad
    \sup_{0 \leq n \leq N} \expect \abs{\approxthetatilde{n} - \vartheta_{n \Delta}}^2
    \leq C \left( \Delta + \sigma^{2\beta} \right),
\end{equation}
where $C = C(T, J)$ is the same constant as in~\eqref{eq:general_discrete_bound}.
In \cref{sec:proof_of_main_results}, for clarity of exposition,
we will in fact first prove the convergence estimate~\eqref{eq:error_bound_simplified_discrete} before showing the more general~\cref{thm:convergence_discrete_equations_full}.
\end{remark}

\begin{remark}
    In the limit $\Delta \to 0$, the error bound~\eqref{eq:general_discrete_bound} becomes
    \begin{align*}
        \sup_{0 \leq n \leq T/\Delta} \expect \abs{\approxtheta{n} - \vartheta_{n \Delta}}^2
          &\leq \left( \sigma^{2\beta} +  \log(1 + \delta^{-1}) \, \delta^2 \right),
    \end{align*}
    which is almost as sharp as the bound obtained in \cref{thm:convergence_continuous_equations}.
    The presence of the extra factor $\log(1 + \delta^{-1})$ in front of $\delta^2$ indicates that it may be possible to obtain a sharper bound.
\end{remark}

\subsection{Accelerating Convergence with Preconditioning}%
\label{sub:accelerating_convergence_using_preconditioning}

In many applications, the condition number of the Hessian of $\Phi_R$ at and around the MAP estimator is very large.
In this situation,
the fastest time scale of~\eqref{eq:averaged_equation},
i.e.\ of gradient descent ($\nu = 0$) or overdamped Langevin ($\nu = 1$) dynamics,
is much smaller than its slowest time scale.
This is evident when $\Phi_R$ is quadratic,
in which case the slowest and fastest time scales correspond to the reciprocals of the smallest and largest eigenvalues of $\hessian \Phi_R$, respectively.
As a result of this wide scale separation,
a very small time step, compared to the time scale of convergence,
is required in order to resolve the dynamics precisely using a numerical method.
For explicit numerical methods,
a wide separation of time scales also leads to a stringent constraints on the time step in order to guarantee stability,
leading to often prohibitive computational costs.

Empirically, we observe -- see \cref{sub:toy_example_with_preconditioning} -- that our multiscale method suffers from a similar issue,
which is not surprising given that~\eqref{eq:alternative_derivative-free} converges to~\eqref{eq:averaged_equation} as $(\delta, \sigma) \to (0, 0)$ by \cref{thm:convergence_continuous_equations}.
This is in contrast with the sampling and inversion methods for inverse problems that are based on the ensemble Kalman filter,
essentially because these methods are affine-invariant~\cite{garbuno2020affine}:
they behave similarly across the class of problems that differ only by an affine transformation.
Ensemble Kalman methods can be viewed, at least in the case of a linear forward model,
as coupled gradient descents dynamics or overdamped Langevin diffusions preconditioned by the covariance of the ensemble,
which provides good stability and convergence properties~\cite{2019arXiv190308866G,2019arXiv191007555C}.

To remedy this issue of overly restrictive constraints on the time step (relatively to the slowest time scales of the problem),
preconditioning can be incorporated in our multiscale method.
More precisely, given a symmetric positive definite matrix $K$,
the dynamics~\eqref{eq:multiscale:evolution_theta} (resp. \eqref{eq:alternative_derivative-free}) can be modified as follows,
\begin{subequations}
\label{eq:multiscale:evolution_theta_preconditioned}%
\begin{align}
    \label{eq:multiscale:evolution_theta_preconditioned_theta}%
    &\quad \dot {\theta}
    = - \frac{1}{J \sigma^2} \sum_{j=1}^{J} \eip*{G(\theta^{(j)}) - G(\theta)}{G(\theta) - y}[\Gamma] \, (\theta^{(j)} - \theta)
    -  C_{K}(\Xi) \Sigma^{-1} (\theta - m) + \nu \sqrt{2 D_{K}(\Xi)} \, \dot {w}, \\
     &\theta^{(j)} = \theta + \sigma \,  \sqrt{K} \xi^{(j)}, \qquad
      \qquad \quad \qquad \qquad \qquad \qquad \qquad \qquad j = 1, \dotsc, J, \\
     &\dot {\xi}^{(j)} = - \frac{1}{\delta^2} \, \xi^{(j)} + \sqrt{\frac{2}{\delta^2}} \, \dot {w}^{(j)},
    \qquad \qquad \xi^{(j)}(0) \sim \mathcal N(0, {I_d}),  \qquad
      j = 1, \dotsc, J,
\end{align}
\end{subequations}
where $C_{K}(\Xi) := \sqrt{K} \, C(\Xi) \sqrt{K}$ and $D_{K}(\Xi) = C_K(\Xi)$ (resp. $D_{K}(\Xi) = K$).
Under the linear approximation
\[
    G(\theta^{(j)}) - G(\theta) \approx (\theta^{(j)} - \theta) \cdot \grad G (\theta),
\]
which is accurate for small $\sigma$,
we can rewrite~\eqref{eq:multiscale:evolution_theta_preconditioned_theta} as
\[
    \dot {\theta}
    \approx - C_{K}(\Xi) \grad \Phi_R (\theta) + \nu \sqrt{2 D_{K}(\Xi)} \dot {w},
\]
which suggests that $\{\theta_t\}_{t \geq 0}$ should converge,
in the limit as $\delta \to 0$ and $\sigma \to 0$, to the preconditioned overdamped Langevin dynamics~\eqref{eq:overdamped_Langevin}.
In order to make this more precise,
notice that if the stochastic process $\{\theta_t, \Xi_t\}_{t \geq 0}$ solves~\eqref{eq:multiscale:evolution_theta_preconditioned},
then $\{u_t, \Xi_t\}_{t \geq 0} := \{\sqrt{K^{-1}}\theta_t, \Xi_t\}_{t \geq 0}$
is equal in law to the solution of~\eqref{eq:multiscale:evolution_theta} (resp.~\eqref{eq:alternative_derivative-free}) with the modified forward model
\[
    \tilde G(u) = G(\sqrt{K} \, u),
    \qquad u \in \real^d,
\]
with the modified initial condition $u_0 = \sqrt{K^{-1}} \theta_0$,
and with the modified prior parameters $\tilde {m} = \sqrt{K^{-1}}m$ and~$\tilde {\Sigma} = \sqrt{K^{-1}} \Sigma \sqrt{K^{-1}}$,
i.e.\ with the prior distribution $\mathcal N(\sqrt{K^{-1}}m, \sqrt{K^{-1}} \Sigma \sqrt{K^{-1}})$.
Here we employed the fact that $\sqrt{C_K(\Xi)} w_t = \sqrt{K} \sqrt{C(\Xi)} w_t$ in law.
In view of this connection, \cref{thm:convergence_continuous_equations,thm:convergence_discrete_equations_full} apply \emph{mutatis mutandis} to the dynamics~\eqref{eq:multiscale:evolution_theta_preconditioned} with $D_K(\Xi) = K$.

We motivate in \cref{sub:toy_example_with_preconditioning} that, when the forward model is linear,
a good preconditioning matrix $K$ is given by the covariance of the Bayesian posterior.
In practice, we observed that preconditioning with the posterior covariance works well also for nonlinear forward models,
provided that the posterior distribution is unimodal.
We emphasize that this approach to preconditioning can be applied both with ($\nu = 1$) and without ($\nu = 0$) noise.
In order to approximate the posterior covariance at a reasonable computational cost,
we employ the gradient-free ALDI (gfALDI) approach proposed in~\cite{2019arXiv190308866G},
which enables generating approximate samples from the Bayesian posterior.
The ALDI method is based on the dynamics
\begin{align}
    \notag
     \dot {\theta}^{(\ell)}
     &= - \frac{1}{L} \sum_{k=1}^{L} \eip*{G(\theta^{(k)}) - \bar G}{G(\theta^{(\ell )}) - y}[\Gamma] (\theta^{(k)} - \bar {\theta}) \\
    \label{eq:eki:evolution_of_thetas}%
     &\qquad - C(\Theta) \Sigma^{-1} (\theta^{(\ell )} - m)
     + \frac{d+1}{L}(\theta^{\ell} - \bar \theta)
    + \sqrt{2 C(\Theta)} \, \dot {w}^{(\ell )}, \qquad \ell  = 1, \dotsc, L.
\end{align}
where $\Theta = (\theta^{(1)}, \dotsc, \theta^{(L)})$ and
\begin{align*}
    \bar {\theta} = \frac{1}{L} \sum_{\ell =1}^{L} \theta^{(\ell )},
    \qquad \bar {G} = \frac{1}{L} \sum_{\ell =1}^{L} G(\theta^{(\ell )}).
\end{align*}
The processes $w^{(\ell)}$ are independent Brownian motions in $\real^L$,
$C(\Theta) = \frac{1}{L} \sum_{\ell=1}^{L} (\theta^{(\ell)} - \bar {\theta}) \otimes (\theta^{(\ell)} - \bar {\theta})$
is the covariance matrix associated with the ensemble
and $\sqrt{C(\Theta)}$ is the $\real^{d \times L}$ matrix defined by
\begin{align*}
    \sqrt{C(\Theta)} = \frac{1}{\sqrt{L}}
    \begin{bmatrix}
        (\theta^{(1)} - \bar {\theta}) & \dotsc & (\theta^{(L)} - \bar {\theta})
    \end{bmatrix}.
\end{align*}
It holds that $C(\Theta) = \sqrt{C(\Theta)} \sqrt{C(\Theta)}^\t$,
which motivates the square root notation.
The first argument in the inner product on the right-hand side of \cref{eq:eki:evolution_of_thetas} is related to consensus,
whereas the second argument measures the mismatch with the observed data.
In practice, the initial ensemble members are drawn independently from the Gaussian prior distribution.
When the Bayesian posterior distribution is close to Gaussian,
it is expected that each particle in the ensemble is an approximate sample from the posterior for sufficiently large times.
Therefore, the expectation of an observable $f$ with respect to $\pi_*$ is approximated from ALDI iterates as
\begin{equation}
    \label{eq:averages}
    \expect_{\pi_*} f
    \approx \frac{1}{NL} \sum_{i=N_0+1}^{N_0+N}\sum_{\ell=1}^{L} \delta_{\hat\theta_i^{(\ell)}},
\end{equation}
where the first $N_0$ iterations can be discarded for reducing the bias from initial conditions,
and where $\{\hat\theta_i^{(\ell)}\}_{i \geq 0}$ is a discrete-time approximation of $\theta^{(\ell)}$
for $\ell \in \{1,\dotsc, L\}$.
The main difference between ALDI and EKS is the presence of the second to last term on the right-hand side of~\eqref{eq:eki:evolution_of_thetas}.
It is shown by means of numerical experiments in~\cite{garbuno2020affine} that this corrective drift term,
which was first identified in~\cite{2019arXiv190810890N},
is crucial for accuracy
when the number of particles $L$ is of the same order of magnitude as the dimension of the state space $d$.

Although EKS and ALDI are self-preconditioned,
these methods can suffer from stability issues for small times when the posterior distribution is far from or much more concentrated than the prior;
this stiffness issue is discussed for the ensemble Kalman--Bucy filter in~\cite{Amezcua14}.
In practice, it is often useful to proceed in two steps:
(i) first run $N_0$ iterations with a small time step,
until the ensemble reaches a region of high posterior probability;
(ii) then use a larger step size for the rest of the simulation,
and use only these iterations for the computation of averages with respect to the posterior distribution,
as in~\eqref{eq:averages}.
The matrix $K$ could also be learned from the EKS, UKS or CBS approaches to approximate sampling,
rather than from ALDI.

Before closing this section,
we note that the local preconditioning approach developed in~\cite{LMS16,MR4248267},
based on localized covariance matrices,
could potentially also be useful for improving the performance of our method.
Once a rough approximation of the posterior has been calculated using~\eqref{eq:multiscale:evolution_theta}, for example,
self-preconditioning could be achieved through a localized covariance matrix
constructed from all the previously generated samples that are in the vicinity of the distinguished particle.
It may be worthwhile to explore this idea in future work,
but in this paper we consider only preconditioning through a position-independent matrix $K$, as in~\eqref{eq:multiscale:evolution_theta_preconditioned}.

\section{Numerical Experiments}%
\label{sec:numerical_experiments}
In this section,
we present numerical experiments demonstrating the performance and limitations of  our method.
\Cref{sub:low_dimensional_parameter_space,sub:toy_example_with_preconditioning,sub:bimodal_example} serve as proof of concept:
in \cref{sub:low_dimensional_parameter_space} a toy inverse problem with low-dimensional parameter and data is considered,
in \cref{sub:bimodal_example} an example with bimodal posterior is considered,
and in \cref{sub:toy_example_with_preconditioning} preconditioning is exemplified.
A more challenging simulated example, closer to inverse problems arising in
real applications, is then considered in \cref{sub:high_dimensional_parameter_space}.

\subsection{Low-dimensional Parameter Space}%
\label{sub:low_dimensional_parameter_space}
We first consider the inverse problem with low-dimensional parameter space that
was first presented in~\cite{MR3400030} and later employed as a test problem in~\cite{HV18,2019arXiv190308866G}.
In this problem,
the forward model maps the unknown $(u_1, u_2) \in \real^2$ to the observation $\bigl(p(x_1), p(x_2)\bigr) \in \real^2$,
where $x_1 = 0.25$ and $x_2 = 0.75$ and
where $p(x)$ denotes the solution to the boundary value problem
\begin{equation}
    \label{eq:boundary_value_problem}%
    \derivative*{1}{x} \left( \e^{u_1} \derivative*{1}[p]{x} \right) = 1, \qquad x \in [0, 1],
\end{equation}
with boundary conditions $p(0) = 0$ and $p(1) = u_2$.
This problem admits the following explicit solution~\cite{HV18}:
\[
    p(x) = u_2 x + \e^{-u_1} \left( - \frac{x^2}{2} + \frac{x}{2} \right).
\]
We employ the same parameters as in~\cite{2019arXiv190308866G}:
the prior distribution is $\mathcal N(0, \sigma^2 {I_2})$ with $\sigma = 10$,
and the noise distribution is $\mathcal N(0, \gamma^2 {I_2})$ with $\gamma = 0.1$.
The observed data is taken to be $y = (27.5, 79.7)$.
Since~\eqref{eq:boundary_value_problem} admits an explicit solution,
the forward model can be evaluated very quickly.
As a result,
obtaining a good approximation of the MAP with our multiscale method takes less than a minute on a personal computer.

We first investigate the performance of the algorithm~\eqref{eq:multiscale:evolution_theta_discrete} when $\nu = 0$,
i.e.\ when an approximation of the MAP estimator is sought.
All the numerical results related to this problem were obtained with $J = 8$ auxiliary processes, with a fixed time step $\Delta = 10^{-3}$,
and with $\approxtheta{0} = (1, 103)$ as the initial condition.

The effect of the parameter $\sigma$, which encodes the radius of exploration around $\theta$,
is illustrated in \cref{fig:1d_elliptic_optimization_sigma}.
In the left panel, we present the trajectories of the solution $\approxtheta{n}$ obtained with~\eqref{eq:multiscale:evolution_theta_discrete} for fixed small $\delta = 10^{-7}$ and several values of $\sigma$.
In the right panel, we present the evolution of the error, in the Euclidean norm, along the trajectories.
In contrast with deterministic algorithms, the iterates produced by our method do not converge to a limit,
which is reflected in the fact that the error oscillates indefinitely at a small value as the simulation progresses.
The reason for this is that there does not exist a value of $\theta$ for which the right-hand side of~\eqref{eq:multiscale:evolution_theta_theta} is zero for all $\Xi \in (\real^d)^J$.
We also notice that a larger value of $\sigma$ seems to increase the convergence speed in the initial stage of the simulation
but it leads to a larger error in the later stages, as the iterates get close to the MAP estimator.
\begin{figure}[ht]
    \centering
    \includegraphics[height=0.36\linewidth]{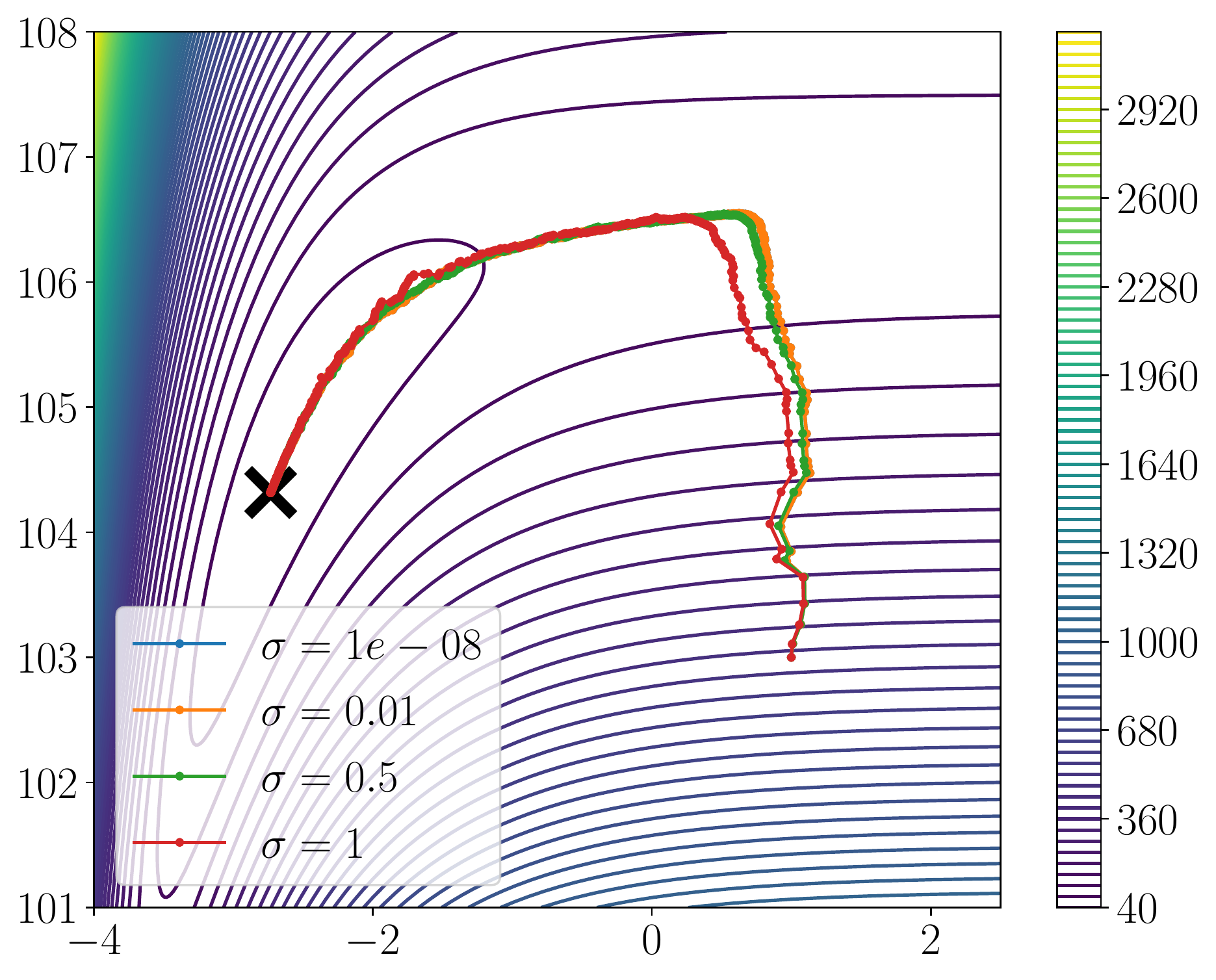}
    \includegraphics[height=0.36\linewidth]{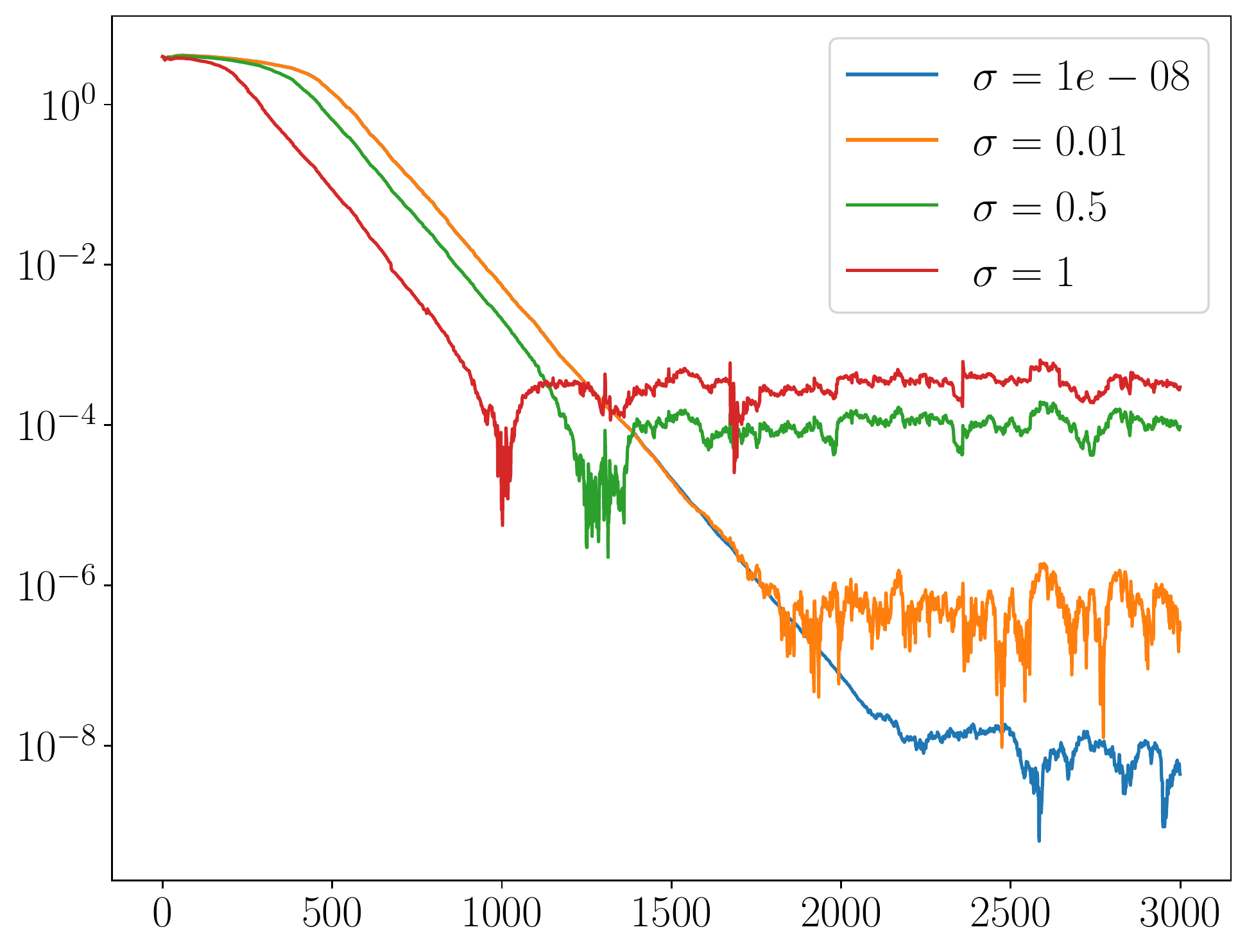}
    \caption{%
        {\bf Left:} Trajectories of the numerical solution $\approxtheta{n}$ for fixed $\delta = 10^{-8}$ and different values of $\sigma$.
        {\bf Right:} Error $\abs*{\approxtheta{n} - \theta_{\rm MAP}}_2$ along the trajectories,
        where $\theta_{\rm MAP}$ is the MAP estimator.
    }%
    \label{fig:1d_elliptic_optimization_sigma}
\end{figure}

The effect of the parameter $\delta$, which influences the correlation between the directions of successive steps,
is illustrated in \cref{fig:1d_elliptic_optimization_delta} for fixed $\sigma = 0.1$.
We observe that the direction of successive steps seems to oscillate more rapidly when $\delta$ is small,
which is consistent with our understanding of the effect of this parameter.
In this particular example, choosing a large $\delta$ does not appear to improve convergence.
\begin{figure}[ht]
    \centering
    \includegraphics[height=0.36\linewidth]{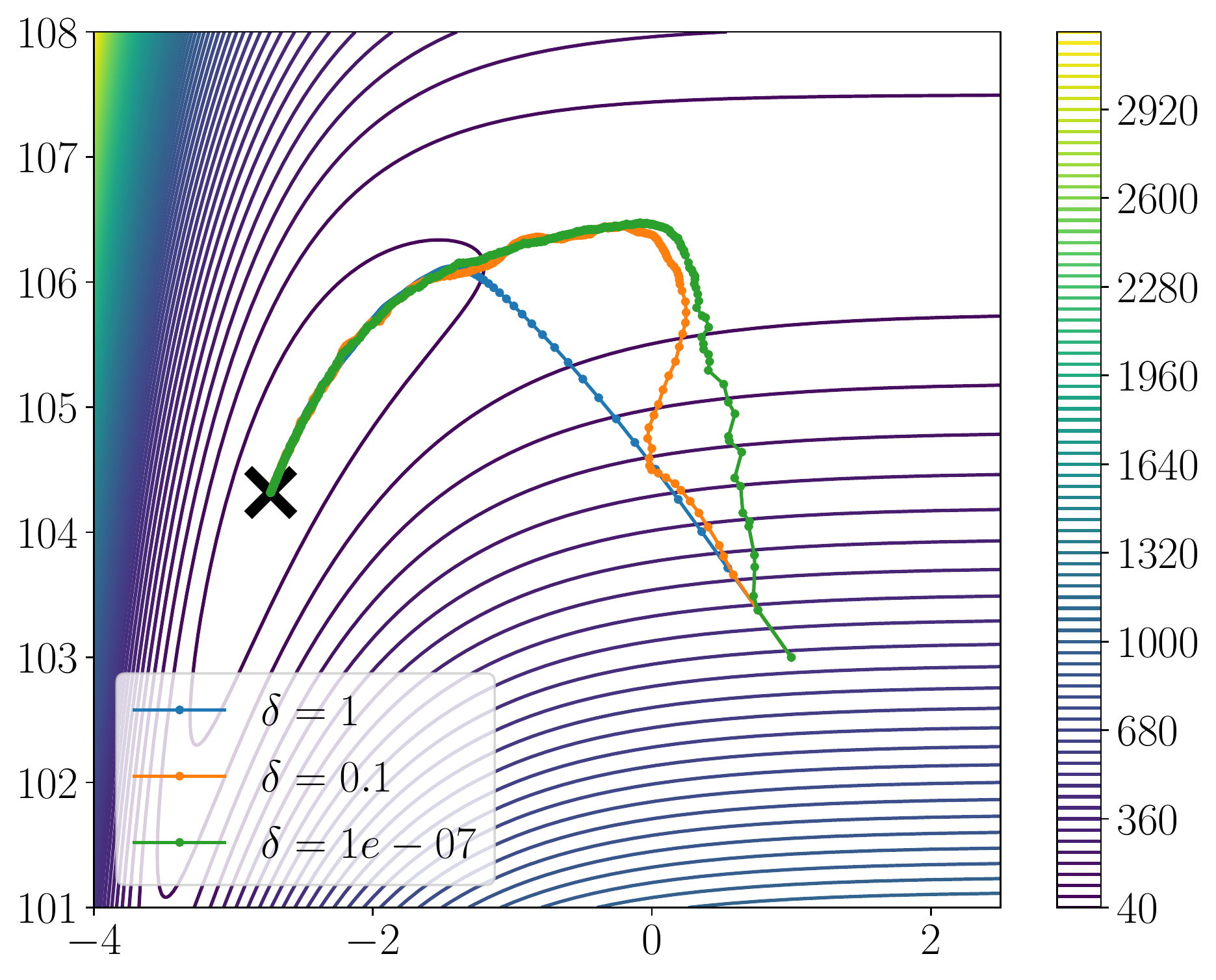}
    \includegraphics[height=0.36\linewidth]{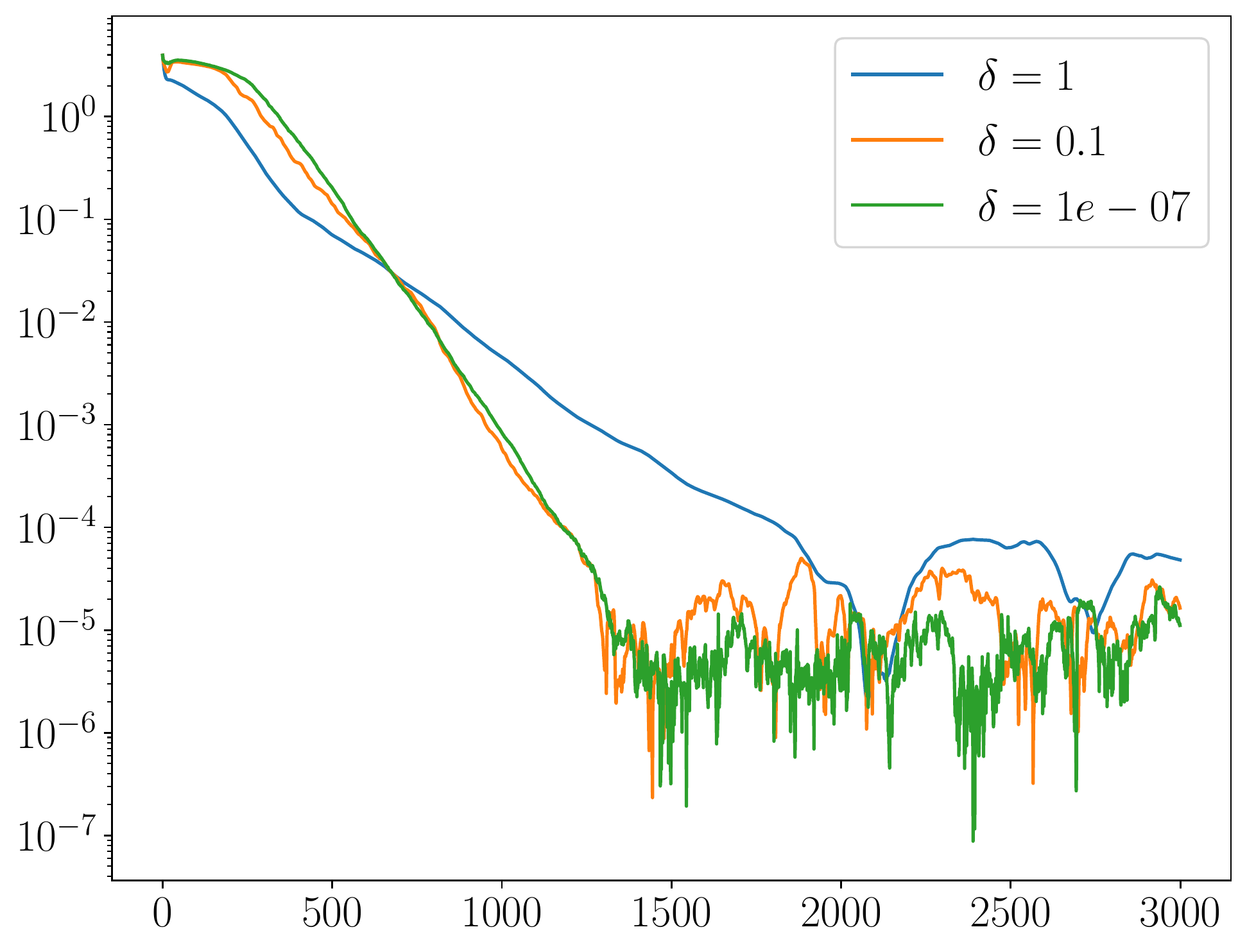}
    \caption{%
        {\bf Left:} Trajectories of the numerical solution $\approxtheta{n}$ for fixed $\sigma = 0.1$ and different values of $\delta$.
        {\bf Right:} Error $\abs*{\approxtheta{n} - \theta_{\rm MAP}}_2$ along the trajectories,
        where $\theta_{\rm MAP}$ is the MAP estimator.
    }%
    \label{fig:1d_elliptic_optimization_delta}
\end{figure}

Let us now investigate the efficiency of~\eqref{eq:multiscale:evolution_theta_discrete} for sampling from the posterior distribution.
For this simulation, we used the parameters $\delta = 10^{-4}$ and $\sigma = 0.01$.
We ran the simulation for 20,000 iterations and, discarding the first 1,000 iterates,
we computed an approximation of the posterior by kernel density estimation with the function \verb?gaussian_kde? from the \verb?scipy.stats? module.
The iterates 1,000 to 20,000, the approximation of the Bayesian posterior based on these iterates, and the true posterior
are depicted in the left, middle and right panels of \cref{fig:1d_elliptic_optimization_noise}, respectively.
It appears from the figure that the approximate posterior is close to the true posterior.
Indeed, the mean and covariance of the true and approximate posterior distributions,
given, respectively, by
\[
    m_p =
    \begin{pmatrix}
        -2.714...  \\
        104.346...
    \end{pmatrix}
    \quad
    C_p =
    \begin{pmatrix}
        0.0129... & 0.0288... \\
        0.0288... & 0.0808...
    \end{pmatrix},
   \]
and
\[
    \tilde{m}_p =
    \begin{pmatrix}
        -2.686... \\
        104.411...
    \end{pmatrix}
    \quad
    \tilde{C}_p =
    \begin{pmatrix}
        0.0147... & 0.0308... \\
        0.0308... & 0.0866...
    \end{pmatrix}
\]
are fairly close.
\begin{figure}[ht]
    \centering
    \includegraphics[height=0.36\linewidth]{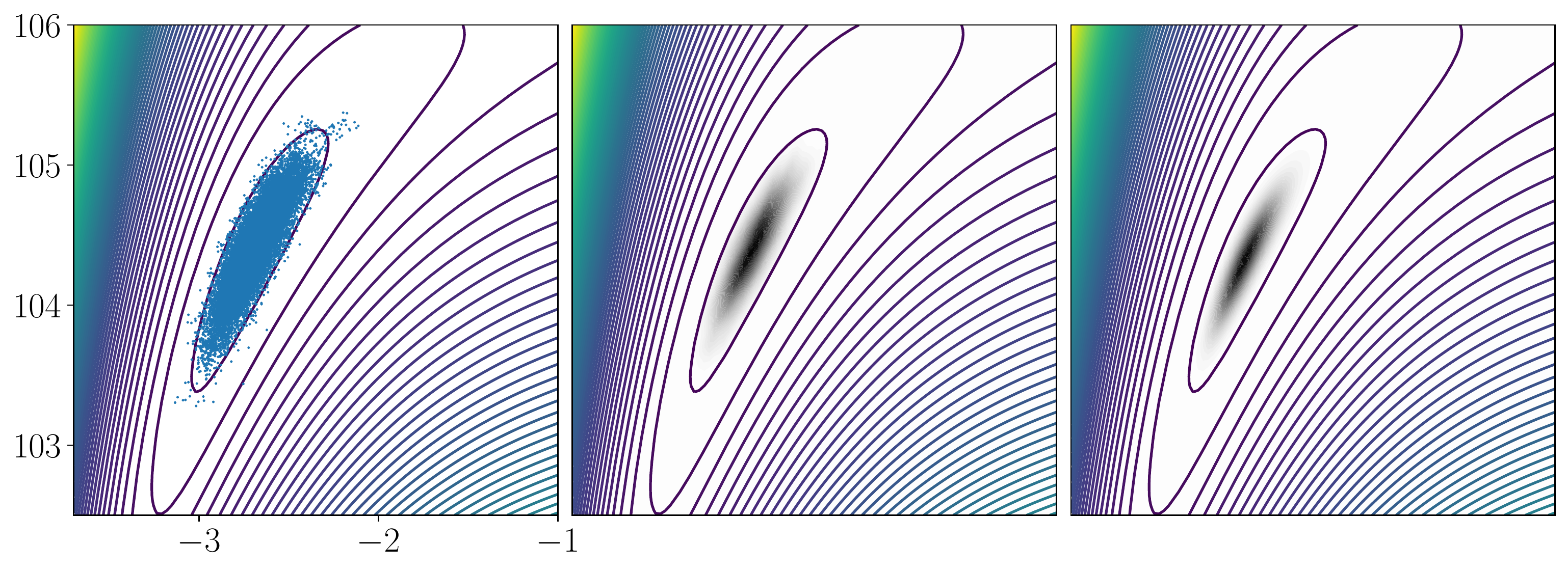}
    \caption{%
    {\bf Left:} Iterates 1,000 to 20,000.
    {\bf Middle:} Approximation of the Bayesian posterior based on these iterates, using kernel density estimation.
    {\bf Right:} True Bayesian posterior.
    }%
    \label{fig:1d_elliptic_optimization_noise}
\end{figure}

\subsection{Two-dimensional bimodal example}%
\label{sub:bimodal_example}
In this section, we consider a bimodal example from~\cite{MR4248267},
associated with forward model
\[
    G: \real^2 \ni \theta \mapsto  \abs{\theta_1 - \theta_2}^2 \in \real,
\]
noise distribution $\eta \sim \mathcal N(0, I_2)$ and prior distribution $\mathcal N(0, I_2)$.
This example is employed in~\cite{MR4248267} for demonstrating the ability of an appropriately localized version of ALDI to sample from multimodal distributions.
We consider two different values for the data: $y=2$ and $y = 4.2297$,
the latter value being the one used in~\cite{MR4248267}.
In both cases, the posterior is bimodal and so ensemble Kalman methods fail to accurately capture the Bayesian posterior,
but the energy barrier between the two modes is much higher when $y = 4.2297$,
which makes this case challenging also for our method.

The approximate posteriors for $y=2$ and $y=4.2297$,
obtained from $10^6$ iterations of our method discretized using~\eqref{eq:multiscale:evolution_theta_discrete} with a time step $\Delta = 10^{-2}$,
parameters $\delta = \sigma = 10^{-5}$, and $J = 8$ auxiliary particles,
are depicted in~\cref{fig:bimodal_example_good,fig:bimodal_example_bad}, respectively.
Whereas gfALDI without localized covariance fails for both values of $y$,
our method gives a very good approximation of the true posterior distribution when $y=2$,
which illustrates the strength of our method compared to ensemble Kalman-based methods without localization.

For $y = 4.2297$, however,
the posterior distribution appears to be well approximated by our method within each of the two high-density regions
but the probabilities of these regions are not accurately captured:
only $33 \%$ of the iterates are such that $\theta_2 - \theta_1 \geq 0$,
whereas this fraction is $50 \%$ under the true posterior probability.
This discrepancy is due to the strong metastability of the dynamics caused by bimodality,
and indeed the line $\theta_2 - \theta_1 = 0$ was crossed by the iterates only~88 times during the simulation.
For comparison, for $y = 2$ the proportion of iterates such that $\theta_2 - \theta_1 \geq 0$ was 50.8\%,
and the line $\theta_2 - \theta_1 = 0$ was crossed 33492 times.
This example shows that strongly multi-modal distributions are challenging for our method,
which is not surprising given the connection with overdamped Langevin dynamics established in~\cref{thm:convergence_continuous_equations}.
\begin{figure}[ht]
    \centering
    \includegraphics[width=0.99\linewidth]{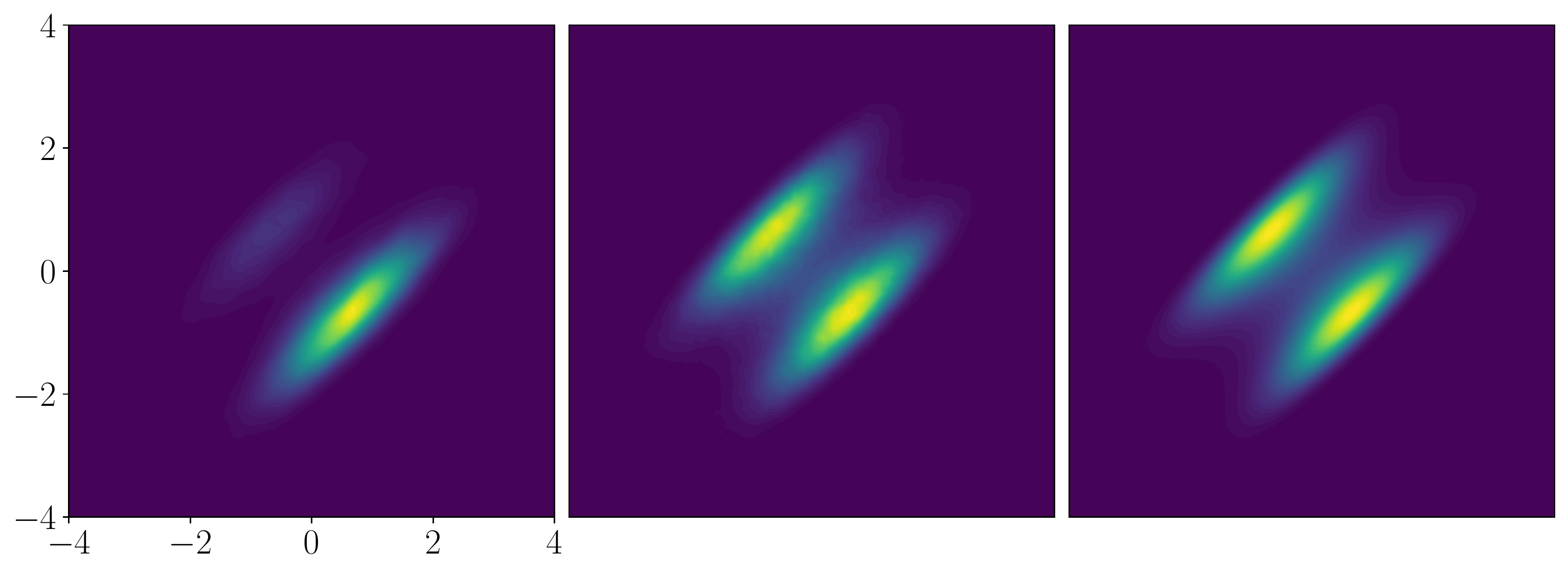}
    \caption{%
        All figures are for $y = 2$.
        {\bf Left:} Approximation of the Bayesian posterior constructed by kernel density estimations from $10^6$ samples generated with gfALDI.
        {\bf Middle:} Approximation of the Bayesian posterior constructed from $10^6$ iterations of the multiscale method.
        {\bf Right:} True Bayesian posterior.
    }%
    \label{fig:bimodal_example_good}
\end{figure}
\begin{figure}[ht]
    \centering
    \includegraphics[width=0.99\linewidth]{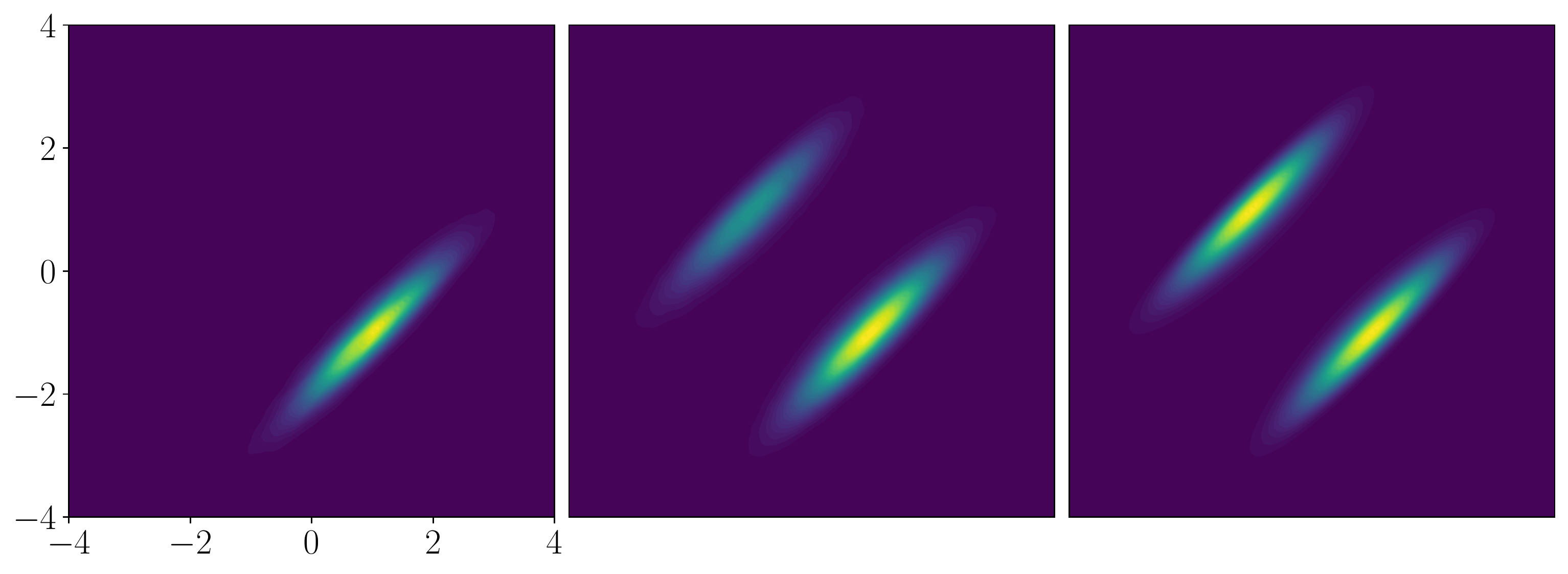}
    \caption{%
        All figures are for $y = 4.2297$.
        {\bf Left:} Approximation of the Bayesian posterior constructed by kernel density estimations from $10^6$ samples generated with gfALDI.
        {\bf Middle:} Approximation of the Bayesian posterior constructed from $10^6$ iterations of the multiscale method.
        {\bf Right:} True Bayesian posterior.
    }%
    \label{fig:bimodal_example_bad}
\end{figure}

\subsection{Toy Example with Preconditioning}%
\label{sub:toy_example_with_preconditioning}

\newcommand{\precond}{\tilde}
We now illustrate the preconditioning methodology proposed in~\cref{sub:accelerating_convergence_using_preconditioning} for a simple inverse problem where the forward model is given by the linear function
$G: \real^3 \ni \theta \mapsto (\theta_1, k \theta_2, k^2 \theta_3)$ with $k = 5$.
We choose the other parameters of the inverse problem as follows: $y = (1, k, k^2)$,
$\Sigma = + \infty {I_3}$ (that is, there is no prior regularization) and $\Gamma = {I_3}$,
so that the MAP estimator is $\theta_{\rm MAP} = (1, 1, 1)$
and the least-squares functional is given by
\begin{align}
    \label{eq:preconditioning_before_precond}
    \Phi(\theta; y) = \frac{1}{2} \left( |\theta_1 - y_1|^2 + k^2|\theta_2 - y_2|^2 + k^4 |\theta_3 - y_3|^2 \right).
\end{align}
The largest eigenvalue of the Hessian of $\Phi(\dummy; y)$ at $\theta_{\rm MAP}$ is equal to $k^4$ so,
if we were to use a gradient descent for~\eqref{eq:preconditioning_before_precond} with the explicit Euler method in order to find the minimizer of $\Phi(\dummy; y)$,
then the constraint on the time step $\Delta$ in order to ensure stability would be that $k^4\Delta < 2$.
Since our method converges to a gradient descent in the limit as $\delta \to 0$ and $\sigma \to 0$,
it is reasonable to expect that a similar constraint should hold to ensure stability of~\eqref{eq:multiscale:evolution_theta_discrete},
and this is indeed what we observed numerically;
in particular, we verified in the case where $\sigma = \delta = 10^{-5}$ that~\eqref{eq:multiscale:evolution_theta_discrete} is stable when $\Delta = 1/k^4$ but unstable when $\Delta = 3/k^4$.

For this problem, the covariance of the Bayesian posterior is given by
\[
    K =
    \begin{pmatrix}
        1 &  0 & 0 \\
        0 &  \frac{1}{k^2} &  0 \\
        0 &  0 &  \frac{1}{k^4}
    \end{pmatrix}.
\]
This is clearly the optimal preconditioner for calculating the MAP estimator.
Indeed, in this case, the limiting equation associated with~\eqref{eq:multiscale:evolution_theta_preconditioned} in the limit as $\max(\delta,\sigma) \to 0$ is
\begin{align}
    \label{eq:preconditioning_after_precond}
    \d \theta_t = - (\theta_t - \theta_{\rm MAP}) \, \d t + \nu \sqrt{2 K} \, \d w_t.
\end{align}
When employed for integrating this equation with $\nu = 0$,
the explicit Euler scheme is stable for time steps satisfying $\Delta < 2$.
In practice, we approximate the preconditioning parameter $K$ from 100,000 iterations of gfALDI,
run with 5 particles initialized independently as $\mathcal N(0, I_3)$ and discretized using the Euler--Maruyama method.
The first 10,000 iterations are run with a small time step $\Delta = 10^{-3}$ to avoid stability issues,
and they are discarded for the computation of the posterior covariance
in order to reduce the bias originating from initial conditions.
The remaining 90,000 iterations are run with a larger time step~$\Delta = 10^{-2}$,
which does not lead to stability issues in the later stages of the simulation
thanks to self-preconditioning by the ensemble covariance.
We employed~$L = 5$ particles because it was shown in~\cite{garbuno2020affine} that
choosing $L > d + 1$ ensures that the continuous-time ALDI dynamics with gradients
(which coincide with gradient-free ALDI dynamics in the case of a linear forward model)
is ergodic with respect to the product measure $\pi_* \times \dotsb \times \pi_*$,
where $\pi_*$ is the Bayesian posterior.
Using this approach,
we obtain a good approximation of the posterior covariance:
\[
    \widehat K =
    \begin{pmatrix}
        1.00...       & 0.00553...  & -0.000302... \\
        0.00553...    & 0.0399...   &  6.48... \times 10^{-5} \\
        -0.000302...  & 6.48... \times 10^{-5}  &  0.00154...
    \end{pmatrix}.
\]

\begin{figure}[b!]
    \centering
    \includegraphics[width=0.65\linewidth]{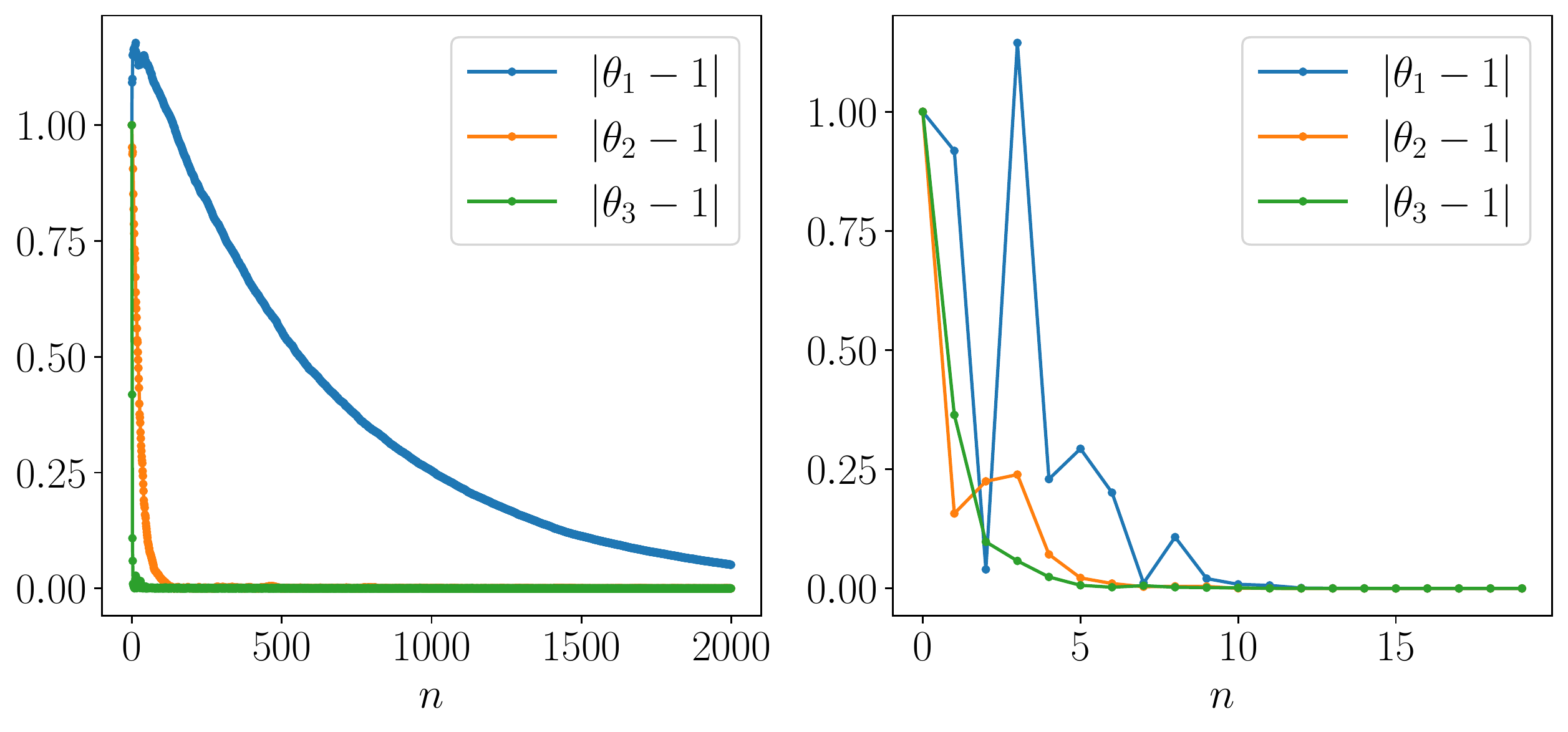}
    \caption{%
        Error between the iterates and the MAP estimator, without (left) and with (right) preconditioning.
        Here the label $n$ of the $x$ axis denotes the iteration index.
        We observe that a better approximation of the MAP estimator is obtained after 20 iterations with preconditioning,
        rather than after 2,000 iterations without preconditioning.
    }%
    \label{fig:preconditioning}
\end{figure}

The effect of using this preconditioner is illustrated in \cref{fig:preconditioning} in the case where~\eqref{eq:multiscale:evolution_theta_discrete} is used in optimization mode,
i.e.\ with $\nu = 0$.
The left and right panels present the evolution of the error with and without preconditioning.
The parameters employed are $\sigma = \delta = 10^{-5}$ and $J = 8$,
and the time step was set to $\Delta = 1$ with preconditioning and $\Delta = 1/k^4$ without preconditioning.
The initial condition taken in both cases was $\theta_0 = (0, 0, 0)^T$.
It appears clearly from the figure that preconditioning significantly accelerates the convergence.

In practice, the mean of the ensemble obtained after application of gfALDI is also useful;
it can be employed as initial condition for~\eqref{eq:multiscale:evolution_theta_preconditioned}.
This is the approach taken in the next section.

\subsection{Higher-dimensional Parameter Space}%
\label{sub:high_dimensional_parameter_space}

We now present an example from~\cite{2019arXiv190308866G} in which the calculation of the forward map
requires the solution to a partial differential equation (PDE) in two dimensions,
and is therefore computationally expensive.
More precisely, we consider the inverse problem of finding the permeability from noisy pressure measurements in a Darcy flow.
This problem falls into the framework developed in~\cite{MR3839555},
and it is natural to model it as an inverse problem with infinite-dimensional parameter space.
In order to be amenable to the numerical methods developed in this paper, however,
the problem needs to be discretized:
this requires defining a finite-dimensional approximation space for the unknown parameter
and specifying a numerical approximation for the calculation of the forward map.
We begin by presenting the inverse problem in its natural infinite-dimensional setting,
and then we give the associated discrete approximation,
which we solve numerically using~\eqref{eq:multiscale:evolution_theta_discrete} together with the preconditioning approach proposed in \cref{sub:accelerating_convergence_using_preconditioning}.

At the infinite-dimensional level,
the abstract inverse problem we consider is that of estimating the logarithm of the permeability profile,
denoted by $a(x)$,
based on noisy measurements of the solution~$p(x)$ to the PDE
\begin{subequations}
\label{eq:pde_darcy}
\begin{align}
    - \grad \cdot \big( \e^{a(x)} \grad p(x) \big) &= f(x), \qquad &x \in D, \\
    p(x) &= 0, &x \in \partial D.
\end{align}
\end{subequations}
Here $D = [0, 1]^2$ is the domain and $f(x) = 50$ represents a source of fluid.
For the prior distribution, we employ a Gaussian measure on $\lp{2}{D}$ with mean zero and precision (inverse covariance) operator given by
\[
    \mathcal C^{-1} = (- \laplacian + \tau^2 \mathcal I)^{\alpha},
\]
equipped with Neumann boundary conditions on the space of mean-zero functions.
The eigenfunctions and eigenvalues associated with the covariance operator are given by
\[
    \psi_{\ell}(x) = \cos \bigl( \pi (\ell_1 x_1 + \ell_2 x_2)\bigr),
    \qquad \lambda_{\ell} = \left(\pi^2 \abs{\ell}^2 + \tau^2\right)^{-\alpha}, \qquad \ell \in \nat^2.
\]
The parameters $\tau$ and $\alpha$ control the characteristic length scale and the smoothness of samples drawn from the prior,
respectively.
For the numerical experiments presented in this section,
we take the same values for these parameters as in~\cite{2019arXiv190308866G}: $\tau = 3$ and $\alpha = 2$.
In this setting,
it can be shown by reasoning as in~\cite[Example 2.19]{MR3839555} that
the log-permeability $a(x)$ is almost surely continuous on the closed set $D$,
so there exists a unique solution $p \in H^1(D)$ to~\eqref{eq:pde_darcy} almost surely.

If $a(x) \sim \mathcal N(0, \mathcal C)$, then $\expect\bigl( \ip{a}{\psi_{\ell}} \, \ip{a}{\psi_{m}}\bigr) = \lambda_{\ell} \, \delta_{\ell, m}$
by definition of the covariance operator and by orthonormality of the eigenfunctions $\{\psi_{\ell}\}_{\ell \in \nat^2}$,
where $\ip{\dummy}{\dummy}$ denotes the inner product in $\lp{2}{D}$.
Since $a(x)$ is almost surely in $\lp{2}{D}$,
we deduce that, almost surely,
\begin{equation}
    \label{eq:karhunen-loeve}
    a  = \sum_{\ell \in \nat^2} \ip{a}{\psi_{\ell}} \psi_{\ell}
    =: \sum_{\ell \in \nat^2} \sqrt{\lambda_{\ell}} \, \theta_{\ell} \, \psi_{\ell},
\end{equation}
where the factors $\{\theta_{\ell}\}_{\ell \in \nat^2}$ are independent $\mathcal N(0, 1)$ random variables.
This is the Karhunen--Loève (KL) expansion,
which can be used as a starting point for the definition of probability distributions in infinite dimensions;
see e.g.~\cite{MR3839555} for more details.

In theoretical works on Bayesian inverse problems of the
type considered in this section, the data are usually
modeled as the values taken by a finite number of continuous
linear functionals $\ell_1, \dotsc, \ell_K$ over $H^1_0(D)$,
when evaluated at the solution $p$ to~\eqref{eq:pde_darcy},
perturbed by additive Gaussian noise.
That is, the forward model maps the unknown permeability $a(\dummy)$ to
\[
    \Big( \ell_1(p), \dotsc, \ell_K(p) \Big) \in \real^K.
\]
In practice, however, we consider that the data consist of pointwise measurements of the solution to~\eqref{eq:pde_darcy}, up to noise.
We assume that these are taken at a finite number equidistant points given by
\begin{equation}
    \label{eq:data_points}
    x_{ij} = \left( \frac{i}{M}, \frac{j}{M} \right), \qquad 1 \leq i, j \leq M - 1.
\end{equation}
Since pointwise evaluation is not a continuous functional on $H^1_0(D)$,
our example deviates here from the framework in~\cite{MR3839555}.
As mentioned in~\cite{2019arXiv190308866G},
pointwise evaluation could in principle be approximated by integration against a narrow mollifier,
which would ensure continuity of the functionals,
but we do not discuss this here.
We take the distance between measurement points equal to $1/10$, i.e.\ $M = 10$,
and we work with the noise distribution $\mathcal N(0, \gamma^2 {I_K})$, with $\gamma = 0.01$ and $K = (M-1)^2$.

In order to approximate the solution to the inverse problem numerically,
we truncate the KL series~\eqref{eq:karhunen-loeve} after a finite number of terms and take this truncated series as the object of inference.
More precisely, we take as unknown the vector of parameters $\theta = \{\theta_{\ell}: \abs{\ell}_{\infty} \leq N \} \in \real^{(N+1)^2}$,
and as prior distribution the Gaussian $\mathcal N(0, {I_d})$, with $d = (N+1)^2$.
For each $\theta \in \real^{(N+1)^2}$, a log-permeability field is constructed by summation as $a(\dummy; \theta) := \sum_{\abs{\ell}_{\infty} \leq N} \sqrt{\lambda_{\ell}} \, \theta_{\ell} \, \psi_{\ell}$,
and the corresponding solution to~\eqref{eq:pde_darcy} is approximated with a finite element method (FEM).
This defines an inverse problem with finite-dimensional parameter space,
which is amenable to~\eqref{eq:multiscale:evolution_theta_discrete} or~\eqref{eq:multiscale:evolution_theta_discrete_simplified}.
In practice, we use $N = 7$,
leading to a state space of dimension 64,
and a FEM using quadratic elements over a regular mesh with 20 elements per direction,
implemented with Gridap~\cite{Badia2020}.
Below, we refer to this inverse problem as the \emph{finite-dimensional inverse problem},
in order to distinguish it from the inverse problem with infinite-dimensional parameter space it aims to discretize.

In order to generate data for the finite-dimensional inverse problem,
we employ the same FEM as is employed in the numerical algorithm for the evaluation of the forward model,
with a true permeability also given by a truncated KL expansion using as many KL terms as in the numerical inference,
i.e.\
\begin{equation}
    \label{eq:non_extended_permeability}
    a^{\dagger}(x) = \sum_{\abs{\ell}_{\infty} \leq N} \sqrt{\lambda_{\ell}} \theta_{\ell}^{\dagger} \psi_{\ell}, \qquad \theta^{\dagger}_{\ell} \sim \mathcal N(0, 1).
\end{equation}
Clearly, this does not provide a sample from $\mathcal N(0, \mathcal C)$,
but it does provide a sample consistent with the prior distribution assumed in the finite-dimensional inverse problem.
The logarithm of the true permeability field, as well as its MAP approximation by~\eqref{eq:multiscale:evolution_theta_discrete} with $\nu = 0$,
is illustrated in \cref{fig:log_of_permeablitiy}.
The preconditioning matrix~$K$ is calculated using the methodology outlined in~\cref{sub:accelerating_convergence_using_preconditioning},
by running 200 iterations of gfALDI with an ensemble size equal to 512;
the first 100 iterations are employed for transitioning from the prior to a rough approximation of the posterior,
using the adaptive time-stepping scheme of~\cite{MR3998631},
and then gfALDI is run with fixed time step for an additional 100 iterations,
and only these iterations are used for the approximation of the posterior distribution.

The MAP estimator calculated from 300 iterations of the multiscale method~\eqref{eq:multiscale:evolution_theta_discrete},
with a fixed time step equal to $\Delta = 0.02$, parameters $\delta = \sigma = 10^{-5}$,
and $J = 8$ auxiliary processes,
is illustrated in~\cref{fig:log_of_permeablitiy}.
It appears from the figure that the MAP estimator obtained is close to the truth;
denoting by $\theta^{\rm MAP}$ the MAP estimator,
we indeed calculate that
\begin{equation}
    \label{eq:error_non_extended_b}
    \frac{\norm{a^\dagger - a^{\rm MAP}}[L^2(D)]}{\norm{a^\dagger}[L^2(D)]}
    = \frac{\sqrt{\sum_{\abs{\ell}_{\infty} \leq N} \lambda_\ell \abs{\theta_{\ell}^\dagger - \theta^{\rm MAP}_{\ell}}^2}}
    {\sqrt{\sum_{\abs{\ell}_{\infty} \leq N} \lambda_\ell \abs{\theta_{\ell}^\dagger}^2}}
    = 0.116...
\end{equation}
showing that the relative error is approximately $12 \%$.
\begin{figure}[ht!]
    \centering
    \includegraphics[width=0.8\linewidth]{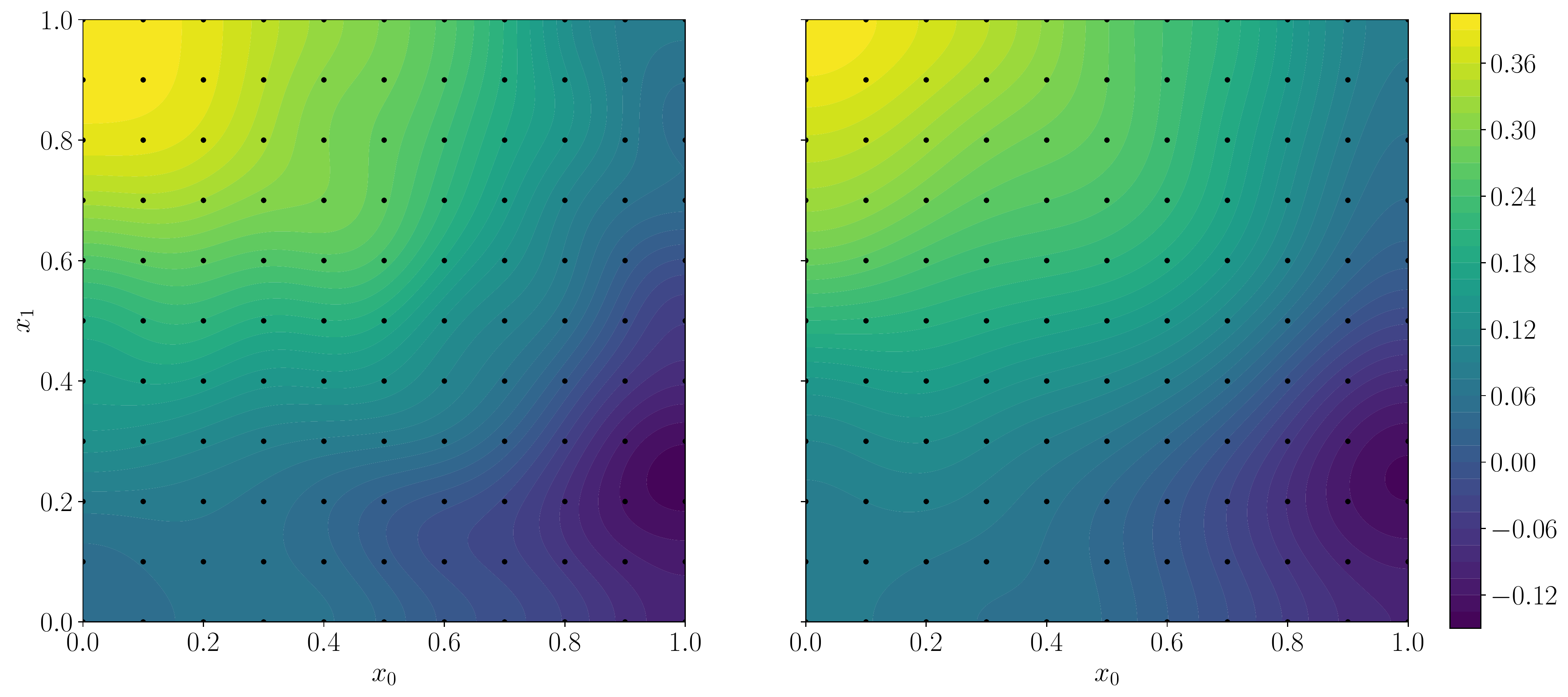}
    \caption{%
        Logarithms of true (left) and approximate (right) permeability profiles.
        The approximate permeability profile was constructed from the approximation of the MAP estimator provided by~\eqref{eq:multiscale:evolution_theta_discrete} with $\nu = 0$.
        The black dots are the observation points.
    }%
    \label{fig:log_of_permeablitiy}
\end{figure}

We now turn our attention to the problem of sampling from the Bayesian posterior.
The marginals of the posterior associated with the first 16 KL coefficients,
obtained by kernel density estimation from 20,000 iterations of the multiscale method~\eqref{eq:multiscale:evolution_theta_discrete}
with $\nu = 1$ and all other parameters unchanged,
are illustrated in~\cref{fig:log_of_permeablitiy_sampling_non_extended}.
In the same figure,
the marginals of the approximate posterior distributions calculated using gfALDI and MCMC are depicted.
The MCMC method employed is a variation on the pCN algorithm described in~\cite{MR3135540}.
Specifically the proposal is based not on the prior Gaussian but on a Gaussian distribution $\mathcal N(m, \alpha K)$;
that is, given $\theta_n$ the proposal for $\theta_{n+1}$ is
\[
    \theta_{n+1}^\star = m + \sqrt{1 - \beta^2} \, (\theta_n - m) + \beta \xi_n,
    \qquad \xi_n \sim \mathcal N(0, \alpha K).
\]
The parameters $m$ and $K$ are set to the mean and covariance of the posterior estimated by gfALDI, respectively,
and a scaling factor $\alpha \geq 1$ is employed to ensure that the posterior distribution is absolutely continuous with respect to $\mathcal N(m, \alpha K)$.
To generate the numerical results in~\cref{fig:log_of_permeablitiy_sampling_non_extended},
we ran 20,000 iterations of this method with $\beta = 0.1$ and $\alpha = 4$.
The agreement between the true parameter and the posterior samples is good overall,
and the agreement between the approximate posteriors is also very good.
\begin{figure}[ht!]
    \centering
    \includegraphics[width=0.8\linewidth]{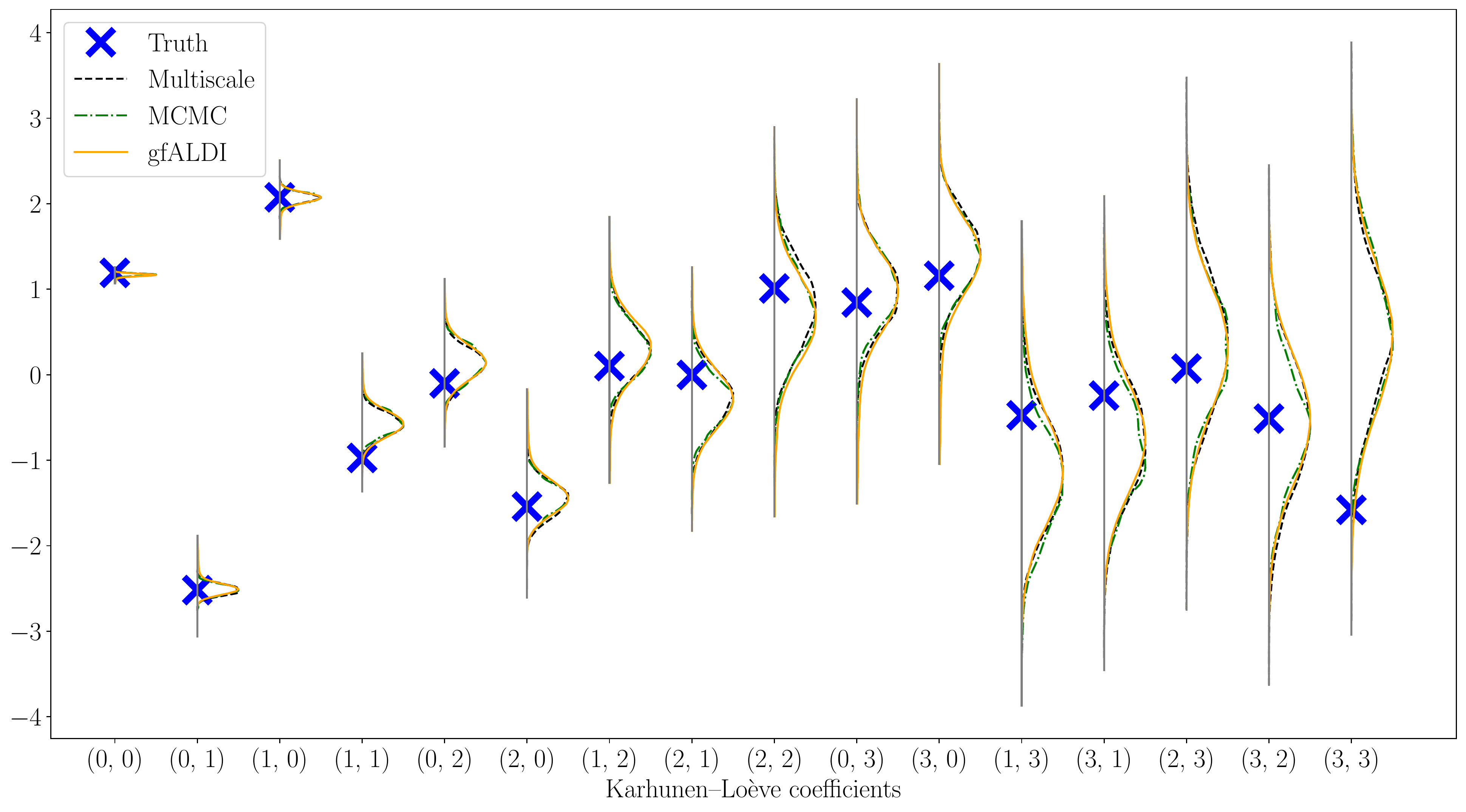}
    \caption{%
        Approximate posterior samples produced by~\eqref{eq:multiscale:evolution_theta_discrete} with $\nu = 1$,
        without model misspecification.
        From these posterior samples,
        the marginal distributions of the KL coefficients were approximated by kernel density estimation using Gaussian kernels;
        they are depicted (non-normalized) in solid lines.
        The crosses are the true values of the coefficients, i.e.\ the values employed to generate the data.
    }%
    \label{fig:log_of_permeablitiy_sampling_non_extended}
\end{figure}

\section{Conclusions and Perspectives for Future Work}%
\label{sec:conculsions}
In this paper,
we introduce a new derivative-free and adjoint-free method for solving Bayesian inverse problems,
specifically for the tasks of sampling from the Bayesian posterior or finding the maximum \emph{a posteriori} estimate.
The method relies on a gradient approximation with a structure similar to that in the ensemble Kalman methods for sampling (EKS, ALDI) and inversion (EKI);
a similar gradient structure was identified within the ``analysis'' step of the EnKF in the paper~\cite{reich2011dynamical}.
In contrast with these algorithms and other approximate sampling methods such as UKS and CBS, however,
the method we propose is provably refineable over a bounded time interval:
using tools from multiscale analysis,
we prove strong pathwise convergence to the gradient descent or overdamped Langevin dynamics,
depending on whether it is employed for inversion or sampling, respectively.
Although we show this result for the particular case where the state space is the $d$-dimensional torus,
we believe that it should be possible,
using results from~\cite{pardoux2001poisson},
to extend our results to the case of an unbounded parameter space.


Since our method is a variation on standard gradient descent,
it suffers from slow convergence when the Bayesian posterior exhibits strong anisotropy or,
relatedly, when the Hessian of the regularized least-squares functional has a large condition number in the part of the domain close to the MAP estimator.
In order to remedy this possible issue,
we propose a preconditioning methodology based on information from ALDI (or EKS, UKS, CBS),
and we demonstrate its efficacy for both inversion and sampling through careful numerical experiments.

Several exciting research avenues remain open for future work.
On the theoretical front, it would be interesting to obtain a uniform-in-time weak error estimate,
both for the continuous-time dynamics and its discrete-time approximation.
This might prove challenging even in the case of a compact parameter space,
because the state space of the auxiliary processes employed for the gradient approximation is unbounded.
Relatedly, it would be useful to obtain a bound,
in terms of the parameters $\sigma$ and $\delta$ and in an appropriate metric,
on the distance between the true Bayesian posterior and that approximated by the method,
i.e.\ the $\theta$-marginal of the invariant measure of~\eqref{eq:multiscale:evolution_theta};
the ideas developed in \cite{bally1996law,bally1996lawb,MattinglyStuartTretyakov2010} might be useful in this regard.
It would also be interesting to study other time discretizations of~\eqref{eq:multiscale:evolution_theta_theta} than the one employed in this paper;
for example, we could consider
discretizations where Bernoulli random variables are employed instead of exact (in law) Brownian increments,
which should not change the weak convergence properties of the method~\cite{MR1214374},
or semi-implicit discretizations (which preserve linearity of the updates) based on the formulation of ensemble Kalman methods in~\cite[Section 4.3.3]{MR3998631};
the diagonally semi-implicit modification of the forward Euler scheme used in~\cite{Amezcua14} in the context of the ensemble Kalman--Bucy filter may also prove useful for developing efficient time-stepping schemes.
Finally, one could study the mean field $J \to \infty$ and averaging $\delta \to 0$ limits for the alternative derivative-free formulation~\eqref{eq:alternative}.

On the practical side,
it will be important to determine how the method can be
coupled to~more efficient or less computationally expensive preconditioners than those computed from ALDI or EKS,
or how these methods, and related methods such as UKS and CBS, can be accelerated.
One may also explore questions related to the parametrization of the multiscale method.
Since a larger value of the parameter $\sigma$ in~\cref{eq:multiscale:evolution_theta} seems to be associated with
faster convergence initially but a larger error later on,
as noted in the description of~\cref{fig:1d_elliptic_optimization_sigma},
it would be interesting to investigate whether a computational gain can be obtained by adapting $\sigma$ during a simulation.
One might, for example, start the dynamics with a relatively large value of~$\sigma$ in order to favor exploration initially,
and then progressively decrease this parameter in order to increase accuracy
once the distinguished particle has reached regions of high posterior probability density.
Finally, it would be interesting to study more precisely and more generally the influence of the parameter~$\delta$,
in order to determine, for example, whether a large value of this parameter can be advantageous for promoting exploration in rugged landscapes.
Our numerical experiments in this paper suggest that choosing $\delta > 0$ may not be advantageous for convergence,
and if this is consistently observed then the simpler discrete time formulation~\eqref{eq:multiscale:evolution_theta_discrete_simplified} corresponding to the case $\delta = 0$ should be preferred over~\eqref{eq:multiscale:evolution_theta_discrete}.

\paragraph{Acknowledgements}%
G.A.P.\ was partially supported by the EPSRC through the grant number EP/P031587/1 and by JPMorgan Chase \& Co under a J.P. Morgan A.I. Research Award 2019.
(Any views or opinions expressed herein are solely those of the authors listed, and may differ from the views and opinions expressed by JPMorgan Chase \& Co. or its affiliates. This material is not a product of the Research Department of J.P. Morgan Securities LLC. This material does not constitute a solicitation or offer in any jurisdiction.)
The work of A.M.S.\ is  supported by NSF (award DMS-1818977) and by the Office of Naval Research (award N00014-17-1-2079).
The work of U.V.\ was partially funded by the Fondation Sciences Mathématiques de Paris (FSMP),
through a postdoctoral fellowship in the ``mathematical interactions'' program.

\newpage
\appendix
\section{Proof of the Main Results}%
\label{sec:proof_of_main_results}

Throughout this section,
we consider that the number of particles $J$ is a fixed parameter.
We often denote the drift in~\eqref{eq:multiscale:evolution_theta_theta} by
\[
    {F^{\sigma}}(\theta, \Xi)
    = - \frac{1}{J \sigma} \sum_{j=1}^{J} \eip*{G(\theta + \sigma \xi^{(j)}) - G(\theta)}{G(\theta) - y}[\Gamma] \, \xi^{(j)}.
\]
We recall that we are working on the multi-dimensional torus $\torus^d$ with a uniform prior, so $\Phi_R = \Phi$.
If the forward model satisfies $G \in C^3(\torus^d, \real^K)$ then, by Taylor's formula,
it holds for all $(\theta, \xi) \in \torus^d \times \real^d$ that
\begin{equation}
\begin{aligned}
    \label{eq:taylor_3}
    &- \frac{1}{\sigma} \eip*{G(\theta + \sigma\xi) - G(\theta)}{G(\theta) - y}[\Gamma] \, \xi \\
    &\qquad\qquad
        = - (\xi \otimes \xi) \grad \Phi_R(\theta)
        - \frac{\sigma}{2} \sum_{k=1}^{K} \sum_{\ell=1}^{K} (\Gamma^{-1})_{k\ell}
        \Bigl( (\xi \otimes \xi) : \hessian G_k(\theta) \Bigr) \Bigl( G_\ell(\theta) - y_\ell \Bigr) \xi \\
    & \qquad \qquad \qquad - \frac{\sigma^2}{6} \sum_{k=1}^{K} \sum_{\ell=1}^{K} (\Gamma^{-1})_{k\ell} \left( (\xi \otimes \xi \otimes \xi) \mathbin{\scalebox{1}[.7]{\vdots}} {\rm D}^3 G_k(\theta + \alpha_k \sigma \xi) \right)
    \Bigl( G_{\ell}(\theta) - y_{\ell} \Bigr) \xi,
\end{aligned}
\end{equation}
for some $\alpha \in [0, 1]^{K}$ depending on $\theta$, $\xi$ and $\sigma$.
Here $\xi \otimes \xi \in \real^{d \times d}$ and $\xi \otimes \xi \otimes \xi \in \real^{d\times d \times d}$ denote
the matrix and third-order tensor with components $(\xi \otimes \xi)_{ij} = \xi_i \xi_j$ and $(\xi \otimes \xi \otimes \xi)_{ijk} = \xi_i \xi_j \xi_k$,
respectively.
We used the notation $\hessian G_k(\theta) \in \real^{d \times d}$ for the Hessian of $G_k$ at $\theta$,
and the notation ${\rm D}^3 G_k(\theta) \in \real^{d \times d \times d}$ for the tensor of third derivatives of $G_k$ at $\theta$,
i.e. the third-order tensor with components $({\rm D}^3 G_k)_{ij\ell}(\theta) = \partial_i \partial_j \partial_\ell G_k(\theta)$.
The symbol $:$ denotes the Frobenius inner product on $\real^{d \times d}$,
and the symbol $\mathbin{\scalebox{1}[.7]{\vdots}}$ denotes the $\real^{d \times d \times d}$ inner product defined by
$S \mathbin{\scalebox{1}[.7]{\vdots}} T = \sum_{i=1}^{d} \sum_{j=1}^{d} \sum_{k=1}^{d} S_{ijk} T_{ijk}$,
for $S, T \in \real^{d \times d \times d}$.

Equation~\eqref{eq:taylor_3} motivates the following notation:
\begin{align*}
    {F_0}(\theta, \Xi) &= - C(\Xi) \grad \Phi_R(\theta), \\
    {F_1}(\theta, \Xi) &= - \frac{1}{2 J} \sum_{j=1}^{J} \sum_{k=1}^{K} \sum_{\ell=1}^{K} (\Gamma^{-1})_{k\ell}
    \Bigl( \bigl(\xi^{(j)} \otimes \xi^{(j)}\bigr) : \hessian G_k(\theta) \Bigr) \Bigl( G_\ell(\theta) - y_\ell \Bigr) \xi^{(j)} \\
    {F_2}(\theta, \Xi) &= \frac{1}{\sigma^2}  \bigl( {F^{\sigma}}(\theta, \Xi) - {F_0}(\theta, \Xi) - \sigma {F_1}(\theta, \Xi) \bigr).
\end{align*}
With this notation, the multiscale system~\eqref{eq:alternative_derivative-free} can be rewritten as
\begin{align*}
    \dot {\theta} &= {F_0}(\theta, \Xi) + \sigma {F_1}(\theta, \Xi) + \sigma^2 {F_2}(\theta, \Xi)+ \nu \sqrt{2} \dot {w}, \\
    \dot {\xi}^{(j)} &= - \frac{1}{\delta^2} \, {\xi}^{(j)} + \sqrt{\frac{2}{\delta^2}} \, \dot {w}^{(j)}, \qquad j = 1, \dotsc, J.
\end{align*}
If $G \in C^2(\torus^d, \real^K)$, then by Taylor's theorem
\begin{align}
    \label{eq:bound_drift_sigma}
    \forall  (\theta, \Xi) \in \torus^d \times (\real^d)^J, \qquad
    \abs{{F^{\sigma}}(\theta, \Xi) - {F_0}(\theta, \Xi)} \leq C \sigma \, \sum_{j=1}^{J} \abs*{\xi^{(j)}}^3,
\end{align}
for a constant $C$ independent of $\sigma$.
Likewise, if $G \in C^3(\torus^d, \real^K)$, then by~\eqref{eq:taylor_3} it holds
\begin{align}
    \label{eq:bound_drift_sigma_squared}
    \forall  (\theta, \Xi) \in \torus^d \times (\real^d)^J, \qquad
    \abs{{F^{\sigma}}(\theta, \Xi) - {F_0}(\theta, \Xi) - \sigma {F_1}(\theta, \Xi)}
    \leq C \sigma^2 \sum_{j=1}^{J} \abs*{\xi^{(j)}}^4,
\end{align}
for a possible different constant $C$, also independent of $\sigma$.
We divide the proof of \cref{thm:convergence_continuous_equations} in two parts.
In the first part, we assume that the forward map satisfies only $G \in  C^{2}(\torus^d, \real^K)$,
and in the second part we obtain a refined estimate for when $G \in  C^{3}(\torus^d, \real^K)$.

\begin{proof}
    [Proof of \cref{thm:convergence_continuous_equations} when $G \in C^{2}(\torus^d, \real^K)$]
    Our approach for this proof is based on~\cite[Chapter 17]{MR2382139}.
    Throughout the proof, $C$ denotes a constant independent of $\delta$ and $\sigma$
    that is allowed to change from occurrence to occurrence.
The generator of the dynamics associated to~\eqref{eq:alternative_derivative-free} is given by
\[
    \mathcal L = \frac{1}{\delta^2} \mathcal L_0 + \mathcal L_1,
\]
where
\begin{align*}
    \mathcal L_0 &= \sum_{j=1}^{J} - \xi^{(j)} \cdot \nabla_{\xi^{(j)}} + \laplacian_{\xi^{(j)}}, \\
    \mathcal L_1 &= - C(\Xi) \, \grad \Phi_R(\theta) \cdot \grad_{\theta} + \bigl({F^{\sigma}}(\theta, \Xi) - {F_0}(\theta, \Xi)\bigr) \cdot \grad_{\theta} + \nu \laplacian_{\theta}.
\end{align*}
Let $\Phi_R^{\varepsilon}$ denote a mollification with parameter~$\varepsilon$ of $\Phi_R$,
as defined in~\cref{lemma:mollifier}.
Since $\Phi \in C^2(\torus^d, \real)$ by the assumption that $G \in C^2(\torus^d, \real^K)$,
it holds by the standard properties of mollifiers that
\begin{equation}
    \label{eq:simple_bound}
    \forall i \in \{1, \dotsc, d\}, \qquad
    \norm{\partial_{\theta_i} \Phi_R^{\varepsilon}}[L^{\infty}(\torus^d)] \leq \norm{\partial_{\theta_i} \Phi_R}[L^{\infty}(\torus^d)],
\end{equation}
and similarly for the second derivatives.
The Poisson equation
\[
    - \mathcal L_0 A(\Xi) =  C(\Xi) - I_{d}  = \frac{1}{J} \sum_{j=1}^{J} (\xi^{(j)} \otimes \xi^{(j)} - {I_d})
\]
admits as unique mean-zero solution
\begin{equation}
    \label{eq:sol_one_poisson}
    A(\Xi) = \frac{1}{2J} \sum_{j=1}^{J} (\xi^{(j)} \otimes \xi^{(j)} - {I_d}).
\end{equation}
Applying It\^o's formula to the function $\psi(\theta, \Xi) =  A(\Xi) \, \grad \textcolor{red}{\Phi_R^{\varepsilon}(\theta)}$,
we obtain
\begin{align*}
    \psi(\theta_t, \Xi_t)  - \psi(\theta_0, \Xi_0)
    &= \frac{1}{\delta^2} \int_{0}^{t} \bigl({I_d} - C(\Xi_s)\bigr) \grad \Phi_R^{\textcolor{red}{\varepsilon}}(\theta_s) \, \d s
    + \int_{0}^{t} \bigl({F^{\sigma}}(\theta_s, \Xi_s) \cdot \grad_{\theta} + \nu \laplacian_{\theta} \bigl)  \psi(\theta_s, \Xi_s) \d s \\
    &\qquad + \sqrt{\frac{2}{\delta^2}} \, \sum_{j=1}^{J}\int_{0}^{t}  (\d w_s^{(j)} \cdot \grad_{\xi^{(j)}}) \psi(\theta_s, \Xi_s)
       + \nu \sqrt{2} \, \int_{0}^{t}  (\d w_s \cdot \grad_{\theta}) \psi(\theta_s, \Xi_s).
\end{align*}
Since the rigorous interpretation of the It\^o SDE~\eqref{eq:alternative_theta} is in integral form,
we have
\[
    \theta_t - \theta_0 = - \int_{0}^{t} C(\Xi_s) \grad \Phi_R(\theta_s) \, \d s + \int_{0}^{t} \bigl({F^{\sigma}}(\theta_s, \Xi_s) - {F_0}(\theta_s, \Xi_s)\bigr) \, \d s + \nu \sqrt{2} w_t,
\]
and so we deduce
\begin{align*}
    \theta_t - \theta_0
    =\, & - \int_{0}^{t} \grad \Phi_R(\theta_s) \, \d s
    + \int_{0}^{t} \bigl(I_d - C(\Xi_s) \bigr) \bigl( \grad \Phi_R(\theta_s) - \grad \Phi_R^{\varepsilon} (\theta_s) \bigr) \, \d s \\
        &+ \int_{0}^{t} \bigl({F^{\sigma}}(\theta_s, \Xi_s) - {F_0}(\theta_s, \Xi_s)\bigr) \, \d s + \nu \sqrt{2} w_t + \delta^2 \bigl( \psi(\theta_t, \Xi_t)  - \psi(\theta_0, \Xi_0)  \bigr) \\
        &- \delta^2 \int_{0}^{t} \bigl({F^{\sigma}}(\theta_s, \Xi_s) \cdot \grad_{\theta} + \nu \laplacian_{\theta} \bigr) \psi(\theta_s, \Xi_s) \, \d s - \delta M_t - \delta^2 N_t
\end{align*}
where $M_t$ and $N_t$ are the martingale terms:
\begin{align*}
    M_t &= \sqrt{2} \, \sum_{j=1}^{J}\int_{0}^{t} \left( \d w_s^{(j)} \cdot \grad_{\xi^{(j)}} \right) \psi(\theta_s, \Xi_s) =: \sum_{j=1}^{J} M_t^{(j)}, \\
    N_t &= \nu \sqrt{2} \,\int_{0}^{t} \left( \d w_s \cdot \grad_{\theta} \right) \psi(\theta_s, \Xi_s).
\end{align*}
Using the fact that $\vartheta$ solves the averaged equation~\eqref{eq:averaged_equation} with the same initial condition and Brownian motion,
we deduce
\begin{align*}
    \label{eq:equation_error}
    \theta_t - \vartheta_t
    =\, & - \int_{0}^{t} \bigl( \grad \Phi_R(\theta_s) - \grad \Phi_R(\vartheta_s) \bigr) \, \d s
    + \int_{0}^{t} \bigl(I_d - C(\Xi_s) \bigr) \bigl( \grad \Phi_R(\theta_s) - \grad \Phi_R^{\varepsilon} (\theta_s) \bigr) \, \d s \\
        &+ \int_{0}^{t} \bigl({F^{\sigma}}(\theta_s, \Xi_s) - {F_0}(\theta_s, \Xi_s)\bigr) \, \d s
        + \delta^2 \bigl( \psi(\theta_t, \Xi_t)  - \psi(\theta_0, \Xi_0)  \bigr) \\
        &- \delta^2 \int_{0}^{t} \bigl({F^{\sigma}}(\theta_s, \Xi_s) \cdot \grad_{\theta} + \nu \laplacian_{\theta} \bigr) \psi(\theta_s, \Xi_s) \, \d s - \delta M_t - \delta^2 N_t.
\end{align*}
Let $e_t = \theta_t - \vartheta_t$.
Using the Lipschitz continuity of $\grad \Phi_R$ in order to bound the first integral on the right-hand side,
the bound~\eqref{eq:bound_mollification} for the second,
the bound~\eqref{eq:bound_drift_sigma} for the third,
the simple inequality~\eqref{eq:simple_bound} for the fourth,
and the bound~\eqref{eq:bound_third_derivatives} for the fifth,
noticing that this bound implies the inequality
\[
    \forall (\theta, \Xi) \in \torus^d \times (\real^d)^J, \qquad
    \max \Bigl\{ \bigl\lvert \grad_{\theta} \psi_i(\theta, \Xi) \bigr\rvert,
    \bigl\lvert \hessian_{\!\!\theta} \psi_i(\theta, \Xi) \bigr\rvert_{\rm F} \Bigr\}
        \leq C \varepsilon^{-1} \biggl( 1 + \sum_{j=1}^{J} \abs*{\xi^{(j)}}^2 \biggr)
\]
for all $i \in \{1, \dotsc, d\}$,
we obtain
\begin{align*}
    \abs{e_t}
    \leq \, &L \int_{0}^{t} \abs{e_s} \d s +
    C \varepsilon \int_{0}^{t} \left(1 + \sum_{j=1}^{J} \abs*{\xi_s^{(j)}}^2 \right) \,\d s
        + C \sigma  \sum_{j=1}^{J} \int_{0}^{t} \abs*{\xi_s^{(j)}}^3 \, \d s
        + C \delta^2  \left( \sum_{j=1}^{J} \abs{\xi_0^{(j)}}^2 + \sum_{j=1}^{J} \abs{\xi_t^{(j)}}^2  \right) \\
            &+ C \delta^2 \varepsilon^{-1} \int_{0}^{t} \left( 1 + \sum_{j=1}^{J} \abs{\xi_s^{(j)}}^5 \right) \, \d s + \delta \abs{M_t} + \delta^2 \abs{N_t}.
\end{align*}
Here $L$ is the Lipschitz constant of $\grad \Phi_R$.
Raising to the power $p$, letting $\varepsilon = \delta$, taking the supremum and taking the expectation, we obtain
\begin{align*}
    &\expect \left( \sup_{0 \leq s \leq t} \abs{e_s}^p \right)
    \leq C T^{p-1} \int_{0}^{t} \expect \abs{e_s}^p \d s +
    C \delta^p T^{p-1} \int_{0}^{t} \biggl(1 + \sum_{j=1}^{J} \expect \abs*{\xi_s^{(j)}}^{2p} \biggr) \,\d s \\
    &\qquad + \sigma^p T^{p-1} \, \sum_{j=1}^{J} \int_{0}^{t} \expect \abs*{\xi_s^{(j)}}^{3p} \, \d s
         + C \delta^{2p}  \biggl( \sum_{j=1}^{J} \expect \abs*{\xi_0^{(j)}}^{2p} + \sum_{j=1}^{J} \expect \left( \sup_{0 \leq s \leq t} \abs*{\xi_s^{(j)}}^{2p} \right)  \biggr) \\
        &\qquad + C \delta^{p} T^{p-1} \int_{0}^{t} \left( 1 + \sum_{j=1}^{J} \expect \abs{\xi_s^{(j)}}^{5p} \right) \, \d s
        + C \delta^p \expect \left( \sup_{0\leq s\leq t}\abs{M_t}^p \right) + C\delta^{2p} \expect \left( \sup_{0\leq s\leq t}\abs{N_t}^p \right) \qquad \forall  t \in [0, T].
\end{align*}
Since $\xi^{(1)}, \dotsc, \xi^{(J)}$ are identically distributed stationary stochastic processes,
their moments are constant in time and they coincide, so we deduce
\begin{align*}
    \expect \left( \sup_{0 \leq s \leq t} \abs{e_s}^p \right)
    \leq& \, C \int_{0}^{t} \expect  \left( \sup_{0 \leq u \leq s} \abs{e_u}^p  \right) \d s +
    C \sigma^p + C \delta^{p} + C \delta^{2p} \expect \left( \sup_{0 \leq s \leq t} \abs*{\xi_s^{(1)}}^{2p} \right)  \\
        & + C \delta^p \expect \left( \sup_{0\leq s\leq t}\abs{M_t}^p \right)+ C\delta^{2p} \expect \left( \sup_{0\leq s\leq t}\abs{N_t}^p \right) \qquad \forall t \in [0, T].
\end{align*}
By \cref{lemma:bound_supremum_ou}, there exists a constant $C$ depending only on $p$ such that
\begin{equation}
    \label{eq:bound_sup}
    \forall \delta > 0,
    \qquad \expect \left(\sup_{0 \leq t \leq T} \abs{\xi_t^{(1)}}^{p} \right) \leq C \left(1 + \log \left(1 + \frac{T}{\delta^2}\right) \right) ^{p/2},
\end{equation}
Therefore, since
\[
    \delta \log \left( 1 + \left(1 + \frac{T}{\delta^2}\right)\right) \xrightarrow[\delta \to 0]{} 0,
\]
it holds
\begin{equation}
    \label{eq:bound_supremum_ou}
    \forall \delta \in (0, 1], \qquad  \delta^p \expect \left(\sup_{0 \leq t \leq T} \abs{\xi_1}^{2p} \right) \leq C.
\end{equation}
Let us now bound the martingale terms.
Using the fact that the summands in the definition of $M_t$ are identically distributed
and using the moment inequality~\cite[Theorem 7.2]{MR2380366}, which is valid for $p \geq 2$,
we obtain
\begin{align*}
    \expect \left( \sup_{0 \leq t \leq T} \abs{M_t}^p \right)
    &\leq \expect \left(J^{p-1}\sum_{j=1}^{J}\sup_{0 \leq t \leq T} \abs{M_t^{(j)}}^p \right)
    = J^p \expect \left(\sup_{0 \leq t \leq T} \abs{M_t^{(1)}}^p \right) \\
    &\leq J^p \left( \frac{p^3}{p-1} \right)^{\frac{p}{2}} T^{\frac{p-2}{p}}
    \int_{0}^{T} \expect \abs{ \grad_{\xi^{(1)}} \psi(\theta_t, \Xi_t) }[\rm F]^p \, \d t,
\end{align*}
where $\abs{\dummy}[\rm F]$ denotes the Frobenius norm.
Since $\grad \Phi_R$ is bounded on $\torus^d$ and the moments of $\xi^{(1)}_t, \dotsc, \xi^{(J)}_t$ are constant in time,
we obtain the bound
\begin{align*}
    \forall p \geq 2, \quad \forall T \geq 0, \qquad
    \expect \left( \sup_{0 \leq t \leq T} \abs{M_t}^p \right)
    &\leq C T^{1 + \frac{p-2}{p}},
\end{align*}
and for $1 < p \leq 2$ we have
\[
    \expect \left( \sup_{0 \leq t \leq T} \abs{M_t}^p \right)
    \leq \expect \left( \sup_{0 \leq t \leq T} \left(1 + \abs{M_t}^2\right) \right) = 1 + \expect \left( \sup_{0 \leq t \leq T} \abs{M_t}^2 \right) \leq 1 + C T.
\]
Similarly, for fixed $T$ it holds
\[
    \forall p > 1, \qquad
    \expect \left( \sup_{0 \leq t \leq T} \abs{N_t}^p \right) \leq C.
\]
Using these bounds together with~\eqref{eq:bound_supremum_ou}, we obtain
\begin{align*}
    \expect \left( \sup_{0 \leq s \leq t} \abs{e_s}^p \right)
    \leq& \, C \int_{0}^{t} \expect  \left( \sup_{0 \leq u \leq s} \abs{e_u}^p  \right) \d s +
    C \sigma^p + C \delta^{p}.
\end{align*}
We then obtain the required bound by Gr\"onwall's inequality, which concludes the proof.
\end{proof}

\begin{proof}
    [Proof of \cref{thm:convergence_continuous_equations} when $G \in C^{3}(\torus^d, \real^K)$]
    The idea of the proof is the same.
    The only difference is that now we consider additionally the Poisson equation
    \[
        - \mathcal L_0 \phi(\Xi; \theta) = - {F_1}(\theta, \Xi).
    \]
    For the sake of simplicity, we consider only the case where $\nu = 0$,
    in which case a regularization in the same spirit as~\eqref{eq:mollification} is not necessary.
    Since the right-hand side is a cubic polynomial in $\xi^{(1)}, \dotsc, \xi^{(J)}$ for fixed $\theta$,
    its average with respect to the invariant measure of $\Xi$ is zero
    and the solution to the equation is itself cubic (with only cubic and linear terms) in the variables $\xi^{(1)}, \dotsc, \xi^{(J)}$.
    Indeed, the eigenfunctions of $\mathcal L_0$ are given by tensor products of Hermite polynomials;
    see, for example,~\cite[Section 4.4]{pavliotis2011applied} and~\cite{AAGPUV_2017}.
    Therefore,
    after applying It\^o's formula to the function~$\varphi(\theta, \Xi) := A(\Xi) \, \grad \Phi_R(\theta) + \phi(\Xi; \theta)$,
    we obtain
    \begin{align*}
        \theta_t - \theta_0
        =\, & - \int_{0}^{t}  \grad \Phi_R(\theta_s) \, \d s + \sqrt{2} w_t + \int_{0}^{t} \bigl({F^{\sigma}}(\theta_s, \Xi_s) - {F_0}(\theta_s, \Xi_s) - \sigma F_1(\theta_s, \Xi_s) \bigr) \, \d s \\
            &+ \delta^2 \bigl( \varphi(\theta_t, \Xi_t)  - \varphi(\theta_0, \Xi_0)  \bigr)
            - \delta^2 \int_{0}^{t} {F^{\sigma}}(\theta_s, \Xi_s) \cdot \grad_{\theta} \varphi(\theta, \Xi) \, \d s \\
            & + \delta \sqrt{2} \, \sum_{j=1}^{J}\int_{0}^{t} \left( \d w_s^{(j)} \cdot \grad_{\xi^{(j)}} \right) \varphi(\theta_s, \Xi_s)
            + \nu \delta^2  \sqrt{2} \,\int_{0}^{t} \left( \d w_s \cdot \grad_{\theta} \right) \varphi(\theta_s, \Xi_s).
    \end{align*}
    By~\eqref{eq:bound_drift_sigma_squared}, the last term on the first line leads to a bound scaling as $\sigma^{2p}$.
    The other terms are bounded as in the proof of \cref{thm:convergence_continuous_equations} in the case where $G \in C^{2}(\torus^d, \real^K)$,
    which is possible because, like $\psi$ in that proof,
    the function $\varphi$ and its derivatives are polynomial functions in the variables $\xi^{(1)}, \dotsc, \xi^{(J)}$.
\end{proof}

\subsection{Analysis of the Discrete-time Numerical Method}%
\label{sub:analysis_of_the_numerical_method}

Before showing \cref{thm:convergence_discrete_equations_full}, we show a preparatory result.
\begin{proposition}
    \label{thm:convergence_discrete_equations}
    Assume that $G \in C^2(\torus^d, \real^K)$ and let $\approxtheta{n}$ be the solution obtained by~\eqref{eq:multiscale:evolution_theta_discrete_simplified}.
    Then there exists a constant $C = C(T, J)$ such that
    \begin{equation*}
        \forall (\sigma, \Delta) \in \real^+ \times \real^+, \qquad
        \sup_{0 \leq n \leq N} \expect \abs{\approxtheta{n} - \vartheta_{n \Delta}}^2
        \leq C \left( \Delta + \sigma^{2\beta} \right),
    \end{equation*}
    where
    \[
        \beta =
        \begin{cases}
            1 &\text{if } G \in C^2(\torus^d), \\
            2 &\text{if } G \in C^3(\torus^d).
        \end{cases}
    \]
\end{proposition}

\begin{proof}
    Our strategy of proof is loosely based on that of~\cite[Theorem 2.4]{weinan2005analysis}.
    Let us denote by $\{\approxvartheta{n}\}_{n=0}^{N}$ the Euler--Maruyama approximation of the solution $\{\vartheta_t\}_{t\in[0,T]}$ to the averaged equation~\eqref{eq:averaged_equation},
    i.e.\ the discrete-time solution obtained from the iteration
    \begin{equation}
        \label{eq:euler_maruyama_averaged}
        \approxvartheta{n+1} =
        \approxvartheta{n} + \grad \Phi_R(\approxvartheta{n}) \, \Delta + \nu \sqrt{2 \Delta} \, x_n,
        \qquad \approxvartheta{0} = \approxtheta{0}.
    \end{equation}
    By the standard theory of numerical methods for SDEs~\cite{MR3097957,MR1214374},
    the difference between $\vartheta_{n \Delta}$ and its approximation $\approxvartheta{n}$ satisfies the bound
    \begin{equation}
        \label{eq:euler_maruyama_error}
        \expect \left(\sup_{0 \leq n \leq T/\Delta} \abs{\approxvartheta{n} - \vartheta_{n \Delta}}^2 \right)
        \leq C \Delta.
    \end{equation}
    Here and below, $C$ denotes a constant independent of $\sigma$ and $\Delta$
    that is allowed to change from occurrence to occurrence.
    Subtracting~\eqref{eq:euler_maruyama_averaged} from the equation for $\approxtheta{n}$~\eqref{eq:multiscale:evolution_theta_discrete_simplified},
    we obtain that the error $e_n := \approxtheta{n} - \approxvartheta{n}$ at step~$n$ satisfies
    \begin{align*}
        e_{n+1}
        &= e_n + \left( {F^{\sigma}}(\approxtheta{n})  + C(\approxXi{n}) \grad \Phi_R(\approxtheta{n}) \right) \Delta
        + \left( \grad \Phi_R(\approxvartheta{n})  - \grad \Phi_R(\approxtheta{n}) \right) \Delta \\
                      & \quad\qquad + \left( \grad \Phi_R(\approxtheta{n}) - C(\approxXi{n}) \grad \Phi_R(\approxtheta{n}) \right) \Delta\\
        &=: e_n + X_n \Delta + Y_n \Delta + Z_n \Delta.
    \end{align*}
    Therefore
    \begin{align*}
        \expect \abs{e_{n+1}}^2 =\,
        &  \expect \abs{e_n + \Delta X_n + \Delta \, Y_n }^2
        + \Delta^2 \expect \abs{Z_n}^2
        + 2 \Delta \expect  \bigl( (e_n + \Delta X_n + \Delta Y_n) \cdot Z_n \bigr).
    \end{align*}
    Using the tower property of conditional expectation, it holds
    \begin{align*}
        \abs{\expect \left(e_n  \cdot Z_n \right)}
        &= \abs{\expect \left( \expect \bigl( e_n  \cdot Z_n  \, | \, \approxtheta{n}, \approxvartheta{n} \bigr)  \right)}
        = \abs{\expect \left( e_n \cdot \expect \bigl( Z_n  \, | \, \approxtheta{n}, \approxvartheta{n} \bigr)  \right)} = 0,
    \end{align*}
    because, by the definition of $Z_n$ and the fact that $\xi_n^{(1)}, \dotsc, \xi_n^{(J)}$ are independent of $\approxtheta{n}$ and $\approxvartheta{n}$,
    it holds $\expect \left( Z_n  \, | \, \approxtheta{n}, \approxvartheta{n} \right)  = 0$.
    Thus, since by Young's inequality $(a + b)^2 \leq (1 + \varepsilon) a^2 + \left(1 + \frac{1}{\varepsilon}\right)b^2$ for any $a, b \in \real$,
    \begin{align}
        \notag
        \expect \abs{e_{n+1}}^2
        &= \expect \abs{e_n + \Delta X_n + \Delta \, Y_n }^2
        + \Delta^2 \expect \abs{Z_n}^2 + 2 \Delta^2 \expect (X_n \cdot Z_n + Y_n \cdot Z_n) \\
        \label{eq:bound_error_numerical_squared}
        &\leq (1 + \Delta)\expect \abs{e_n}^2 + \left(2 \Delta + 3 \Delta^2 \right) (\expect \abs{X_n}^2 + \expect \abs{Y_n}^2)
        + 3 \Delta^2 \expect \abs{Z_n}^2.
    \end{align}
    We now bound the terms one by one.
    \begin{itemize}
        \item
            By~\eqref{eq:bound_drift_sigma},
            it holds
            \begin{align*}
                \forall (\theta, \Xi) \in \torus^d \times (\real^d)^J, \qquad
                \abs{{F^{\sigma}} (\theta, \Xi) + C(\Xi) \grad \Phi_R(\theta)}
                \leq C \sigma \sum_{j=1}^{J} \abs{\xi^{(j)}}^3.
            \end{align*}
            Taking the expectation and using the fact that the moments of $\approxxi{n}{1}, \dotsc, \approxxi{n}{J}$ are constant in $n$,
            it holds $\expect \abs{X_n}^2 \leq C \sigma^2$.
            If $G$ is three times differentiable,
            we can carry out the Taylor expansion to the next order as in \eqref{eq:bound_drift_sigma_squared},
            leading to the refined bound $\expect \abs{X_n}^2 \leq C \sigma^4$.
        \item
            The expectation of $\abs{Y_n}^2$ can be bounded by using the Lipschitz continuity of $\Phi_R$:
            \[
                \expect \abs{Y_n}^2 = \expect \abs{ \grad \Phi_R(\approxvartheta{n})  - \grad \Phi_R(\approxtheta{n}) }^2
                \leq C \expect \abs{e_n}^2.
            \]
        \item
            To bound $\expect \abs{Z_n}^2$,
            we use the fact that $\Phi_R$ is uniformly bounded and that the moments of $\xi^{(1)}_n, \dotsc, \xi^{(1)}_n$ are constant in $n$:
            \begin{equation}
                \label{eq:bound_zs}
                \expect \abs{Z_n}^2 \leq
                2 \expect \abs{\grad \Phi_R(\approxtheta{n})}^2
                + 2 \expect \abs{C(\approxXi{n}) \grad \Phi_R(\approxtheta{n})}^2
                \leq C.
            \end{equation}
    \end{itemize}

    Going back to~\eqref{eq:bound_error_numerical_squared} and combining the bounds, we deduce
    \begin{align*}
        \forall \delta \leq 1, \qquad
        \expect \abs{e_{n+1}}^2
        &\leq (1 + \Delta)\expect \abs{e_n}^2 + C \Delta (\sigma^{2\beta} + \expect \abs{e_n}^2) + C \Delta^2 \\
        &\leq (1 + C \Delta) \expect \abs{e_n}^2 + C \Delta \left(\sigma^{2\beta} + \Delta \right).
    \end{align*}
    Let $\varepsilon = \sigma^{2\beta} + \Delta$.
    Applying the previous bound recursively,
    \begin{align*}
        \expect \abs{e_{n+1}}^2
        &\leq (1 + C \Delta)\bigl((1 + C \Delta)\expect \abs{e_{n-1}}^2  + C \Delta \varepsilon\bigr)  + C \Delta \varepsilon \\
        &\leq \dots \leq (1 + C \Delta)^{n+1} \, \expect \abs{e_{0}}^2  + C \Delta \varepsilon\sum_{i=0}^{n} (1 + C \Delta)^{i} \\
        &\leq \expect \abs{e_{0}}^2 \, \e^{C T} + (n\Delta) \e^{C T} C \varepsilon \leq C (\expect \abs{e_{0}}^2 + \varepsilon).
    \end{align*}
    Since $\expect \abs{e_0}^2 = 0$, and in view of~\eqref{eq:euler_maruyama_error}, this concludes the proof.
\end{proof}

We now prove the more general~\cref{thm:convergence_discrete_equations_full},
which concerns the discrete-time dynamics~\eqref{eq:multiscale:evolution_theta_discrete}.
To this end, it is useful to introduce, for $j = 1, \dotsc, J$ and $n = 0, \dotsc, N$,
an approximation $\approxupsilon{n}{j}$ of $\approxxi{n}{j}$ in~\eqref{eq:multiscale:evolution_theta_discrete_xis} such that
the processes~$\approxupsilon{\dummy}{j}$ have a compactly-supported autocorrelation function.
Notice first that, with the same notation as in~\eqref{eq:multiscale:evolution_theta_discrete}, it holds
\begin{align*}
    \approxxi{n}{j}
    &= \approxxi{0}{j}\e^{-n \frac{\Delta}{\delta^2}} + \sqrt{1 - \e^{-2 \frac{\Delta}{\delta^2}}} \sum_{m=1}^{n} x_{m-1}^{(j)} \e^{-(n-m) \frac{\Delta}{\delta^2}}.
\end{align*}
In order to work with iterates that are uncorrelated when far apart in time,
it is natural to define the approximation
\begin{equation}
    \label{eq:definition_zeta}
    \approxupsilon{n}{j} =
    \begin{cases}
        \approxxi{n}{j} \quad &\text{if $n \leq M$} \\
        \sqrt{1 - \e^{-2 \frac{\Delta}{\delta^2}}} \sum_{m=n-M}^{n} x_{m-1}^{(j)} \e^{-(n-m) \frac{\Delta}{\delta^2}}
        \quad &\text{if $n > M$}.
    \end{cases}
\end{equation}
We denote the collection $(\approxupsilon{n}{1}, \dotsc, \approxupsilon{n}{J})$ by $\approxUpsilon{n}$.
The following lemma,
proved in \cref{sec:proof_of_auxiliary_lemma},
is useful in the proof of \cref{thm:convergence_discrete_equations_full} below.
\begin{lemma}
    \label{lemma:auxiliary_result_aprrox_zeta}
    Let $\{\approxupsilon{n}{j}\}_{n=0}^N$ and $\{\approxxi{n}{j}\}_{n=0}^N$, for $j = 1, \dotsc, J$,
    be the discrete-time processes obtained by~\eqref{eq:multiscale:evolution_theta_discrete_xis} and~\eqref{eq:definition_zeta}.
    Then the following bound holds for a constant $C$ independent of $\delta$:
    \begin{align}
        \label{eq:bound_modified_xis}
        \forall j \in \{1, \dotsc, J\}, \quad \forall  n \in \{0, \dotsc, N\}, \qquad
        \expect \abs{\approxxi{n}{j} - \approxupsilon{n}{j}}^4
        &\leq  C \e^{- 4 (M+1) \frac{\Delta}{\delta^2}}.
    \end{align}
    Consequently, it holds
    \begin{align}
        \label{eq:bound_difference_matrices}
        \forall  n \in \{0, \dotsc, N\}, \qquad
        \expect \abs{C(\approxUpsilon{n}) - C(\approxXi{n})}^2
        \leq C \e^{- 2 (M+1)\frac{\Delta}{\delta^2}},
    \end{align}
    where we used the notation $C(\approxUpsilon{n}) = \frac{1}{J} \sum_{j=1}^{J} \approxupsilon{n}{j} \otimes \approxupsilon{n}{j}$.
\end{lemma}
\begin{proof}
    [Proof of \cref{thm:convergence_discrete_equations_full}]
    Throughout the proof, $C$ denotes a constant independent of $\delta$, $\sigma$ and $\Delta$,
    allowed to change from occurrence to occurrence.
    Let $e_n = \approxtheta{n} - \approxvartheta{n}$,
    where $\approxvartheta{n}$ is as defined in~\eqref{eq:euler_maruyama_averaged}.
    In view of~\eqref{eq:euler_maruyama_error}, it is sufficient to obtain a bound on $e_n$ with the same right-hand side as in~\eqref{eq:general_discrete_bound}.
    Using the same definition for $X_n$ and $Y_n$ as in the proof of \cref{thm:convergence_discrete_equations},
    we have
    \begin{align*}
        e_{n+1} = e_n + X_n \Delta + Y_n \Delta + Z_n \Delta + W_n \Delta,
    \end{align*}
    with now
    \begin{align*}
        Z_n &=  C(\approxUpsilon{n}) \grad \Phi_R(\approxtheta{n}) - C(\approxXi{n}) \grad \Phi_R(\approxtheta{n}) ,
        \qquad C(\approxUpsilon{n}) = \frac{1}{J} \sum_{j=1}^{J} \approxupsilon{n}{j} \otimes \approxupsilon{n}{j},\\
        W_n &=  \grad \Phi_R(\approxtheta{n}) - C(\approxUpsilon{n}) \grad \Phi_R(\approxtheta{n}) .
    \end{align*}
    Using the fact that $e_0 = 0$, we deduce
    \begin{align*}
        \abs{e_{n}}^2 &\leq 4\Delta^2 \abs{\sum_{m=0}^{n-1} X_m}^2 + 4\Delta^2 \abs{\sum_{m=0}^{n-1} Y_m}^2 + 4\Delta^2 \abs{ \sum_{m=0}^{n-1} Z_m}^2 + 4\Delta^2 \abs{ \sum_{m=0}^{n-1} W_m}^2\\
                            &\leq 4n\Delta^2 \sum_{m=0}^{n-1} \abs{X_m}^2 + 4n\Delta^2 \sum_{m=0}^{n-1} \abs{Y_m}^2 + 4n\Delta^2 \sum_{m=0}^{n-1} \abs{Z_m}^2 + 4\Delta^2 \abs{ \sum_{m=0}^{n-1} W_m}^2.
    \end{align*}
    Consequently, for $n \leq T/\Delta$,
    \begin{align}
        \label{eq:total_error_discrete_refined}
        \sup_{0 \leq m \leq n}  \expect \abs{e_{m}}^2
          &\leq 4T \Delta \sum_{m=0}^{n-1} \bigl( \expect \abs{X_m}^2 + \expect \abs{Y_m}^2 + \expect \abs{Z_m}^2\bigr)
          + 4\Delta^2 \sup_{0 \leq m \leq n} \left( \expect \abs{ \sum_{\ell=0}^{m-1} W_\ell}^2\right).
    \end{align}
    The first term is bounded as in the proof of~\cref{thm:convergence_discrete_equations}.
    For the second term, we use the Lipschitz continuity of $\Phi_R$ to deduce
    \[
        \expect \abs{Y_m}^2 = \expect \abs{ \grad \Phi_R(\approxtheta{m})  - \grad \Phi_R(\approxvartheta{m}) }^2
        \leq C \expect \abs{e_m}^2
        \leq C \sup_{0 \leq \ell \leq m} \expect \abs{e_{\ell}}^2 .
    \]
    For the third term, we use~\eqref{eq:bound_difference_matrices}, which gives
    \begin{align*}
        \expect \abs{Z_n}^2
        &\leq C \expect \abs{C(\approxUpsilon{n}) - C(\approxXi{n})}^2
        \leq C \e^{- 2 (M+1)\frac{\Delta}{\delta^2}}.
    \end{align*}
    In order to bound the last term of~\eqref{eq:total_error_discrete_refined}, we calculate
    \begin{align*}
        \expect \abs{ \sum_{\ell=0}^{m-1} W_\ell}^2
                &= \sum_{\ell=0}^{m-1} \sum_{k=0}^{m-1} \expect (W_\ell \cdot W_k) \\
                &= \sum_{\ell=0}^{m-1} \expect \abs{W_{\ell}}^2
                + 2 \sum_{\ell=0}^{m-1} \sum_{k=\ell + 1}^{(\ell + M) \wedge (m -1)} \expect (W_\ell \cdot W_k)
                + 2 \sum_{\ell=0}^{m-1} \sum_{k=\ell + M+1}^{m-1} \expect (W_\ell \cdot W_k) \\
                &\leq \sum_{\ell=0}^{m-1} \expect \abs{W_{\ell}}^2
                + 2 M \sum_{\ell=0}^{m-1} \expect \abs{W_{\ell}}^2
                + 2 \sum_{\ell=0}^{m-1} \sum_{k=\ell + M+1}^{m-1} \expect (W_\ell \cdot W_k).
    \end{align*}
    Since $\grad \Phi_R$ is bounded and the moments of $\approxupsilon{n}{j}$ are bounded uniformly in $n$ and $j$,
    we can bound the first and second sums,
    which leads to
    \begin{align}
        \label{eq:intermediate_equation}
        \expect \abs{ \sum_{\ell=0}^{m-1} W_\ell}^2
                &\leq C m (1 + 2M) +
                2 \sum_{\ell=0}^{m-1} \sum_{k=\ell + M + 1}^{m-1} \expect (W_\ell \cdot W_k)
    \end{align}
    We now bound the second term uniformly for $\ell$ and $k$ satisfying $k \geq \ell + M + 1$.
    Using the tower property of conditional expectation,
    we notice that
    \begin{align*}
        \abs{\expect (W_\ell \cdot W_k)}
        &= \abs{\expect \bigl(\expect (W_\ell \cdot W_k\, | \, \approxtheta{\ell}, \approxUpsilon{\ell})\bigr)}
        = \abs{\expect \bigl(W_\ell \cdot \expect (W_k \, | \, \approxtheta{\ell}, \approxUpsilon{\ell})\bigr)} \\
        &= \abs{\expect \bigl(W_\ell \cdot \expect (W_k\, | \, \approxtheta{\ell}, \approxUpsilon{\ell})\bigr)}
        \leq \sqrt{\expect \abs{W_\ell}^2 \expect \abs{\expect(W_k \, | \, \approxtheta{\ell}, \approxUpsilon{\ell})}^2}
        \leq C \sqrt{\expect \abs{\expect(W_k \, | \, \approxtheta{\ell}, \approxUpsilon{\ell})}^2}.
    \end{align*}
    Using the notations $\expect_{\ell} = \expect(\dummy \, | \, \approxtheta{\ell}, \approxUpsilon{\ell})$ and $\expect_\ell(\dummy_1 \, | \, \dummy_2) = \expect(\dummy_1 \,|\, \approxtheta{\ell}, \approxUpsilon{\ell}, \dummy_2)$ for conciseness,
    we obtain
    \begin{align}
        \notag
        \expect_\ell(W_k)
        &= \frac{1}{J} \sum_{j=1}^{J}\expect_\ell \biggl( \Bigl(I_d - \approxupsilon{k}{j} \otimes \approxupsilon{k}{j}\Bigr) \grad \Phi_R(\approxtheta{k}) \biggr)
        = \frac{1}{J} \sum_{j=1}^{J}\expect_\ell \biggl( \expect_\ell \Bigl(  \bigl(I_d - \approxupsilon{k}{j} \otimes \approxupsilon{k}{j}\bigr) \grad \Phi_R(\approxtheta{k}) \, | \, \approxupsilon{k}{j} \Bigr) \biggr) \\
        \notag
        &= \frac{1}{J} \sum_{j=1}^{J}\expect_\ell \biggl( \Bigl(I_d - \approxupsilon{k}{j} \otimes \approxupsilon{k}{j}\Bigr) \expect_\ell \left( \grad \Phi_R(\approxtheta{k}) \, | \, \approxupsilon{k}{j} \right) \biggr) \\
        \label{eq:intermediate_convergence_discrete}
        &= \frac{1}{J} \sum_{j=1}^{J} \expect_\ell  \left(I_d - \approxupsilon{k}{j} \otimes \approxupsilon{k}{j}\right)  \expect_\ell \left( \grad \Phi_R(\approxtheta{k}) \right) \\
        \notag
        &\quad + \frac{1}{J} \sum_{j=1}^{J}  \expect_\ell \biggl( \Bigl(I_d - \approxupsilon{k}{j} \otimes \approxupsilon{k}{j}\Bigr) \Bigl( \expect_\ell \bigl( \grad \Phi_R(\approxtheta{k}) \, | \, \approxupsilon{k}{j} \bigr) - \expect_\ell \bigl( \grad \Phi_R(\approxtheta{k}) \bigr)  \Bigr) \biggr) .
    \end{align}
    Now note that, for any $k > \ell + M$ and any $j \in \{1, \dotsc, J\}$,
    the random variable $\approxupsilon{k}{j}$ is independent of~$\approxtheta{\ell}$ and~$\approxUpsilon{\ell}$,
    and it has distribution
    \[
        \mathcal N\Bigl(0,  \bigl(1 - \e^{-2(M+1) \frac{\Delta}{\delta^2}} \bigr) \, {I_d} \Bigr).
    \]
    Consequently, we can calculate the first expectation in~\eqref{eq:intermediate_convergence_discrete} exactly.
    Using H\"older's inequality for the other term,
    we obtain
    \begin{align*}
        \abs{\expect_\ell(W_k)}
        \leq
        C\e^{-2(M+1)\frac{\Delta}{\delta^2}} \,+\,
        C \sum_{j=1}^{J} \expect_\ell  \abs{ \expect_\ell \left(\grad \Phi_R(\approxtheta{k}) \,|\, \upsilon^{(j)}_k\right) - \expect_\ell \left(\grad \Phi_R(\approxtheta{k})\right)  }^2.
    \end{align*}
    Therefore, employing Jensen's inequality, we deduce
    \begin{align}
        \label{eq:bound_last_term}
        \expect \abs{\expect_\ell(W_k)}^2
        &\leq
        C\e^{-4(M+1)\frac{\Delta}{\delta^2}} \,+\,
        C \sum_{j=1}^{J} \expect  \abs{ \expect_\ell \left(\grad \Phi_R(\approxtheta{k}) \,|\, \upsilon^{(j)}_k\right) - \expect_\ell \left(\grad \Phi_R(\approxtheta{k})\right)  }^4 \\
        \notag
        &=: C\e^{-4(M+1)\frac{\Delta}{\delta^2}} + \, C \sum_{j=1}^{J} \expect(B_j).
    \end{align}
    Since $\expect_\ell \bigl(\grad \Phi_R(\approxtheta{k-M}) \,|\, \approxupsilon{k}{j}\bigr) = \expect_\ell \bigl(\grad \Phi_R(\approxtheta{k-M})\bigr)$, it holds
    \begin{align*}
        B_j
        &\leq 2^3 \abs{ \expect_\ell \left(\grad \Phi_R(\approxtheta{k}) \,|\, \approxupsilon{k}{j}\right) - \expect_\ell \left(\grad \Phi_R(\approxtheta{k-M}) \,|\, \approxupsilon{k}{j}\right)}^4
        + 2^3 \abs{ \expect_\ell \left(\grad \Phi_R(\approxtheta{k-M})\right) - \expect_\ell \left(\grad \Phi_R(\approxtheta{k})\right)  }^4 \\
        &= 2^3 \abs{ \expect_\ell \left(\grad \Phi_R(\approxtheta{k}) - \grad \Phi_R(\approxtheta{k-M}) \,|\, \approxupsilon{k}{j}\right)}^4
        + 2^3 \abs{ \expect_\ell \left(\grad \Phi_R(\approxtheta{k-M}) - \grad \Phi_R(\approxtheta{k})\right)  }^4.
    \end{align*}
    Using Jensen's inequality and the Lipschitz continuity of $\grad \Phi_R$, we deduce
    \begin{align*}
        B_j
        &\leq C \expect_\ell \left(\abs{\approxtheta{k} - \approxtheta{k-M}}^4 \,|\, \approxupsilon{k}{j}\right)
        + C \expect_\ell \abs{ \approxtheta{k-M} -  \approxtheta{k}}^4,
    \end{align*}
    so $\expect (B_j) \leq C \expect \abs{ \approxtheta{k-M} -  \approxtheta{k}}^4$.
    Using the bound
    \[
        \expect \abs{\approxtheta{k} -  \approxtheta{k-M}}^4
        = \expect \abs{\sum_{i=k-M}^{k-1} F^{\sigma}(\approxtheta{i}, \approxXi{i}) \Delta + \sum_{i=k-M}^{k-1} \nu\sqrt{2 \Delta} x_i}^4
        \leq C \, (M^4 \Delta^4 + M^2 \Delta^2),
    \]
    which is justified because $F^{\sigma}(\theta, \Xi) \leq C \left(1 + \sum_{j=1}^{J}\abs{\xi^{(j)}}^3\right)$ by~\eqref{eq:bound_drift_sigma},
    we deduce by going back to~\eqref{eq:bound_last_term} that
    \begin{align*}
        \expect \abs{\expect_\ell(W_k)}^2
        &\leq
        C\e^{-4(M+1)\frac{\Delta}{\delta^2}} \,+\, C M^4 \Delta^4 + C M^2 \Delta^2.
    \end{align*}
    Taking the square root and returning to the usual notation,
    we obtain
    \begin{align*}
        \sqrt{\expect \abs{\expect(W_k \, | \, \approxtheta{\ell}, \approxUpsilon{\ell})}^2}
        \leq         C\e^{-2(M+1)\frac{\Delta}{\delta^2}} \,+\, C M^2 \Delta^2 + C M \Delta.
    \end{align*}
    Employing this bound in~\eqref{eq:intermediate_equation},
    we deduce
    \begin{align*}
        \Delta^2 \expect \abs{ \sum_{\ell=0}^{m-1} W_\ell}^2
                &\leq C \Delta(1 + 2M) + C \left( \e^{-2(M+1)\frac{\Delta}{\delta^2}} + M^2 \Delta^2 + M \Delta \right) \\
                &\leq C \left( \Delta + M \Delta + M^2 \Delta^2 + \e^{-2(M+1)\frac{\Delta}{\delta^2}} \right).
    \end{align*}
    Letting $M = \lfloor \log(1 + \delta^{-1}) \, \frac{\delta^2}{\Delta} \rfloor$,
    we obtain
    \begin{align}
        \notag
        \Delta^2 \expect \abs{ \sum_{\ell=0}^{m-1} W_\ell}^2
                &\leq C \left(\Delta + \log(1 + \delta^{-1}) \delta^2  + \left(\log (1 + \delta^{-1})\right)^2\delta^4 + \frac{1}{\abs*{1 + \delta^{-1}}^2} \right) \\
                \label{eq:theorem_intermediate_estimate}%
                &\leq  C  \bigl(\Delta + \log(1 + \delta^{-1}) \delta^2 \bigr) \qquad \forall  \delta \in (0, 1].
    \end{align}
    Here we used that, by concavity of the logarithm,
    \[
        \log(1 + \delta^{-1}) = \log(1 + \delta^{-1}) - \log (1) \geq \frac{\delta^{-1}}{1 + \delta^{-1}} \geq \frac{\delta^{-2}}{|1 + \delta^{-1}|^2},
    \]
    so
    \[
        \log(1 + \delta^{-1}) \delta^2 \geq \frac{1}{\abs{1 + \delta^{-1}}^2}.
    \]
    Note that the second term in~\eqref{eq:theorem_intermediate_estimate} vanishes in the limit $\delta \to 0$,
    i.e.\ when $\approxxi{n}{j}$ are drawn independently from $\mathcal N(0, {I_d})$ at each iteration.
    In this case, we recover the statement of \cref{thm:convergence_discrete_equations}.

    Combining everything in~\eqref{eq:total_error_discrete_refined}, we obtain
    \begin{align*}
        \sup_{0 \leq m \leq n}  \expect \abs{e_{m}}^2
          &\leq C \left( \Delta + \sigma^{2\beta}  +  \log(1 + \delta^{-1}) \delta^2
          +  \Delta \sum_{m=0}^{n-1} \left( \sup_{0 \leq \ell \leq m}  \expect \abs{e_{\ell}}^2  \right)\right).
    \end{align*}
    By the discrete Gr\"onwall lemma,
    we deduce
    \begin{align*}
        \sup_{0 \leq n \leq N}  \expect \abs{e_{n}}^2
          &\leq C \e^{CT} \left( \Delta + \sigma^{2\beta} +  \log(1 + \delta^{-1}) \delta^2 \right),
    \end{align*}
    which concludes the proof.
\end{proof}

\section{Auxiliary Results for \texorpdfstring{\cref{thm:convergence_continuous_equations}}{the main Convergence Theorem}}%
\label{sec:auxiliary_results}
\begin{lemma}
    \label{lemma:bound_supremum_ou}
    Let $\delta > 0$ be a fixed parameter and
    let  $X_t$ denote the solution to the scalar Ornstein--Uhlenbeck equation with stationary initial condition,
    \[
        \d X_t = - \frac{1}{\delta^2} \, X_t \, \d t + \sqrt{\frac{2}{\delta^2}} \, \d W_t, \qquad X_0 \sim \mathcal N(0, 1).
    \]
    It holds
    \[
        \forall p \in (1, \infty),
        \qquad \expect \left( \sup_{0 \leq t \leq T} |X_t|^p \right) \leq C \, \left(1 + \log\left(1 + \frac{T}{\delta^2}\right)\right)^{p/2},
    \]
    for a constant $C$ independent of $T$.
\end{lemma}
\begin{proof}
    The strategy of the proof parallels that in~\cite[Theorem A.1]{MR2029590}, so here we give only a sketch.
    We use the notations $W(t)$ and $W_t$ interchangeably.
    The process $\{X_t\}$ is equal in law to $\{Y_{t/\delta^2}\}$,
    where $Y_t$ is the solution to
    \[
        \d Y_t = - Y_t \, \d t + \sqrt{2} \, \d W_t, \qquad Y_0 \sim \mathcal N(0, 1).
    \]
    so we can assume without loss of generality that $\delta = 1$.
    The process $\{Y_t\}_{t\geq 0}$ is equivalent in law to the process $Z_t = \e^{-t} W(\e^{2t})$;
    see, for example,~\cite[Chapter 1]{pavliotis2011applied}.
    Therefore, it holds
    \begin{align*}
        \expect \left( \sup_{0 \leq t \leq T} |Y_t|^p \right)
        &= \expect \left( \sup_{0 \leq s \leq \e^{2T}} \abs{\frac{W_s}{\sqrt{s}}}^p \right) \\
        &\leq \left(\log \log (2 + \e^{2T})\right)^{p/2} \, \expect \left( \sup_{0 \leq s \leq \e^{2T}} \abs{\frac{W_s}{\sqrt{s \log \log (2 + s)}}}^p \right) \\
        &= \left(\log \log (2 + \e^{2T})\right)^{p/2}  \, \expect \abs{N\left(\e^{2T}\right)}^p,
    \end{align*}
    where, for $S \geq 0$,
    \begin{align*}
        N(S) =  \sup_{0 \leq s \leq S} \abs{\frac{W_s}{\sqrt{s \log \log (2 + s)}}}.
    \end{align*}
    Now clearly $\log(\e^{2T} + \, 2) \leq \log(\e^{2 T}) + 2 = 2 T + 2$,
    so $\log \log (2 + \e^{2T}) \leq 1 + \log(1 + T)$.
    It is shown in the proof of~\cite[Theorem A.1]{MR2029590}, and in the references therein,
    that there are constants~$C$ and~$\sigma$ such that
    \[
        \forall S > 0, \quad \forall \lambda > 0,
        \qquad \proba \bigl(N(S) > \lambda\bigr) \leq C \e^{- \frac{\lambda^2}{4 \sigma^2}}.
    \]
    Therefore
    \begin{align*}
        \expect \abs{N(\e^{2T})}^p
        &= \int_{0}^{\infty} \proba \bigl(\abs{N(\e^{2T})}^p > \lambda \bigr) \, \d \lambda \\
        &= \int_{0}^{\infty} \proba \bigl(N(\e^{2T}) > \lambda^{1/p} \bigr) \, \d \lambda
        \leq \int_{0}^\infty C \e^{- \frac{\lambda^{2/p}}{4 \sigma^2}} \, \d \lambda \leq K,
    \end{align*}
    for $K$ independent of $S$,
    which concludes the proof.
\end{proof}

\begin{lemma}
    \label{lemma:mollifier}
    Assume $\Phi_R \in C^2(\torus^d, \real)$,
    and let $\Phi_R^{\varepsilon}$
    denote a mollification with parameter~$\varepsilon$ of $\Phi_R$,
    that~is
    \begin{equation}
        \label{eq:mollification}
        \Phi_R^{\varepsilon} = \varrho_{\varepsilon} * \Phi_R,
        \qquad \varrho_{\varepsilon} = \varepsilon^{-d} \varrho\bigl( \varepsilon^{-1} \theta \bigr),
        \qquad \varrho\colon \real^d \to \real; \theta \mapsto
        \begin{cases}
            k\exp \left( - \frac{1}{1 - \abs{\theta}^2} \right), & \text{if } \abs{\theta} \leq 1, \\
            0, & \text{if } \abs{\theta} > 1, \\
        \end{cases}
    \end{equation}
    where $k$ is a constant such that $\varrho$ integrates to 1 over $\real^d$
    and $*$ is the usual convolution of functions on $\real^d$
    (identifying $\Phi_R$ with a 1-periodic function over $\real^d$).
    Then there is $C$ independent of $\varepsilon$ such that
    \begin{equation}
        \label{eq:bound_third_derivatives}
        \forall (i,j,k) \in \{1, \dotsc, d\}^3, \qquad
        \lVert \partial_{\theta_i} \partial_{\theta_j} \partial_{\theta_k} \Phi_R^\varepsilon(\theta) \rVert_{L^{\infty}(\torus^d)}
        \leq C \varepsilon^{-1} \sup_{\theta \in \torus^d} \abs{\hessian \Phi_R(\theta)}[\rm F],
    \end{equation}
    where $\abs{\dummy}[\rm F]$ denotes the Frobenius norm,
    and
    \begin{equation}
        \label{eq:bound_mollification}
        \norm{\grad \Phi_R^{\varepsilon} - \grad \Phi}[L^{\infty}(\torus^d)]
        \leq C \varepsilon.
    \end{equation}
\end{lemma}
\begin{proof}
    By the standard properties of mollifiers,
    it holds that
    \begin{equation}
        \label{eq:derivative_convolution}
        \forall (i,j,k) \in \{1, \dotsc, d\}^3, \qquad
        \partial_{\theta_i} \partial_{\theta_j} \partial_{\theta_k} \Phi_R^\varepsilon = \partial_{\theta_i} \varrho_{\varepsilon} * \partial_{\theta_j} \partial_{\theta_k} \Phi_R.
    \end{equation}
    We calculate
    \begin{align*}
        \partial_{\theta_i} \varrho_{\varepsilon}(\theta)
        &= - 2 \varepsilon^{-(d+1)} \left( \varrho(\varepsilon^{-1}\theta) \frac{\varepsilon^{-1} \theta_i}{(1 - \lvert \varepsilon^{-1} \theta \rvert^2)^2} \right)
        =: - 2 \varepsilon^{-(d+1)} g_i(\varepsilon^{-1} \theta).
    \end{align*}
    The function $g_i$, for any $i \in \{1, \dotsc, d\}$,
    is smooth and supported in the closed ball of radius 1,
    and so there is some constant $M$ (independent of $\varepsilon$) such that $\lvert g_i(\tau) \rvert \leq M \varrho(\tau/2)$ for all $\tau \in \real^d$.
    Therefore
    \[
        \forall \theta \in  \torus^d, \qquad
        \bigl\lvert \partial_{\theta_i} \varrho_{\varepsilon}(\theta) \bigr\rvert
        = 2 \varepsilon^{-(d+1)} \bigl\lvert g_i(\varepsilon^{-1} \theta) \bigr\rvert
        \leq 2M \varepsilon^{-(d+1)} \varrho(\varepsilon^{-1} \theta/2) = 2^{d+1}M \varepsilon^{-1} \varrho_{2\varepsilon} (\theta),
    \]
    and so, going back to~\eqref{eq:derivative_convolution},
    it is simple to bound the third derivatives of $\Phi_R^{\varepsilon}$ using the second derivatives of~$\Phi_R$,
    at the expense of a large factor $\varepsilon^{-1}$ on the right-hand side:
    \begin{equation*}
        \forall \theta \in \torus^d, \qquad
        \lvert \partial_{\theta_i} \partial_{\theta_j} \partial_{\theta_k} \Phi_R^\varepsilon(\theta) \rvert
        \leq C \varepsilon^{-1} \sup_{\theta \in \torus^d} \abs{\hessian \Phi_R(\theta)}[\rm F].
    \end{equation*}
    This proves~\eqref{eq:bound_third_derivatives}.
    For the second claim, note that the following inequality holds for any 1-periodic Lipschitz continuous $f$ with Lipschitz constant $L$:
    \[
        \forall \theta \in \torus^d, \qquad
        \left\lvert f(\theta) - \varrho_{\varepsilon} * f(\theta) \right\rvert
        = \left\lvert \int_{\real^d}  \bigl( f(\theta) - f(\theta + \tau) \bigr) \, \varrho_{\varepsilon}(\tau) \, \d \tau \right\rvert
        \leq \int_{\real^d} \lvert f(\theta) - f(\theta + \tau) \rvert \varrho_{\varepsilon}(\tau) \, \d \tau
        \leq L \varepsilon.
    \]
    In particular, since $\grad \Phi_R$ is Lipschitz continuous by the assumption,
    we have~\eqref{eq:bound_mollification}.
\end{proof}

\section{Proof of \texorpdfstring{\cref{lemma:auxiliary_result_aprrox_zeta}}{the Auxiliary Lemma}}%
\label{sec:proof_of_auxiliary_lemma}

If $n \leq M$, then $\approxupsilon{n}{j}$ and $\approxxi{n}{j}$ coincide and the statement is true.
If $n > M$, then by definition it holds
\begin{align*}
    \notag
    \approxxi{n}{j} - \approxupsilon{n}{j}
    &= \approxxi{0}{j}\e^{-n \frac{\Delta}{\delta^2}} +
        \sqrt{1 - \e^{-2 \frac{\Delta}{\delta^2}}} \sum_{m=1}^{n-M-1} x_{m-1}^{(j)} \e^{-(n-m) \frac{\Delta}{\delta^2}}\\
    &= \e^{-(M+1) \frac{\Delta}{\delta^2}} \left( \approxxi{0}{j}\e^{- (n-M-1) \frac{\Delta}{\delta^2}} +
    \sqrt{1 - \e^{-2 \frac{\Delta}{\delta^2}}} \sum_{m=1}^{n-M-1} x_{m-1}^{(j)} \e^{-(n-M-1-m) \frac{\Delta}{\delta^2}} \right).
\end{align*}
Using the inequality $\abs{a + b}^4 \leq 8 \abs{a}^4 + 8 \abs{b}^4$ for all $a, b \in \real^d$,
together with the working assumption that $n > M$,
we deduce
\begin{align}
    \label{eq:intermediate_equation_lemma}
    \expect \abs{\approxxi{n}{j} - \approxupsilon{n}{j}}^4
    \leq \e^{-4(M+1) \frac{\Delta}{\delta^2}} \left( 8 \expect \abs{\approxxi{0}{j}}^4 +
    8 \left(1 - \e^{-2 \frac{\Delta}{\delta^2}}\right)^2 \, \expect \abs{\sum_{m=1}^{n-M-1} x_{m-1}^{(j)} \e^{-(n-M-1-m) \frac{\Delta}{\delta^2}}}^4 \right).
\end{align}
It remains to prove that the second term in the round brackets is bounded from above independently of~$\delta$.
Since it holds
$
    \eucnorm{a - b}^4 \leq C \sum_{i=1}^{d} \abs{a_i - b_i}^4,
$
for all $a, b \in \real^d$,
we assume without loss of generality that $d = 1$ in order to establish~\eqref{eq:bound_modified_xis}.
For simplicity of notation, let $v_m = x_{n-M-2 -m}^{(j)}$,
so that
\[
\expect \abs{\sum_{m=1}^{n-M-1} x_{m-1}^{(j)} \e^{-(n-M-1-m) \frac{\Delta}{\delta^2}}}^4
= \expect \abs{\sum_{m=0}^{n-M-2} v_m \e^{-m \frac{\Delta}{\delta^2}}}^4.
\]
Expanding the sum and using the independence of $v_0, \dotsc, v_{n-M-2}$, we calculate
\begin{align*}
    \expect \abs{\sum_{m=0}^{n-M-2} v_m \e^{-m \frac{\Delta}{\delta^2}}}^4
    &= \sum_{m=0}^{n-M-2} \expect \abs{v_m}^4 \e^{-4 m \frac{\Delta}{\delta^2}}
    + 6 \sum_{m=0}^{n-M-2} \sum_{\ell=m+1}^{n-M-2} \expect \abs{v_m}^2 \expect \abs{v_\ell}^2 \e^{-2 (m + \ell) \frac{\Delta}{\delta^2}} \\
    &\leq 3 \sum_{m=0}^{\infty} \e^{-4 m \frac{\Delta}{\delta^2}}
    + 6 \sum_{m=0}^{\infty} \sum_{\ell=m+1}^{\infty} \e^{-2 (m + \ell) \frac{\Delta}{\delta^2}} \\
    &\leq \frac{3}{1 - \e^{-4 \frac{\Delta}{\delta^2}}}
    + \frac{6}{\left(1 - \e^{-2 \frac{\Delta}{\delta^2}}\right)^2}
    \leq \frac{9}{\left(1 - \e^{-2 \frac{\Delta}{\delta^2}}\right)^2}.
\end{align*}
Therefore we deduce~\eqref{eq:bound_modified_xis},
because the denominator cancels out with the factor of the second term in the round brackets in~\eqref{eq:intermediate_equation_lemma}.
In order to derive~\eqref{eq:bound_difference_matrices},
we use H\"older's inequality and~\eqref{eq:bound_modified_xis}:
    \begin{align*}
        \expect \abs{C(\approxUpsilon{n}) - C(\approxXi{n})}^2
        &\leq C \sum_{j=1}^{J} \expect \abs{\approxupsilon{n}{j} \otimes \approxupsilon{n}{j} - \approxxi{n}{j} \otimes \approxxi{n}{j}}^2 \\
        &\leq C \sum_{j=1}^{J} \expect \abs{\approxupsilon{n}{j} \otimes (\approxupsilon{n}{j} - \approxxi{n}{j}) - (\approxxi{n}{j} - \approxupsilon{n}{j}) \otimes \approxxi{n}{j}}^2 \\
        &\leq C \sum_{j=1}^{J} \expect \left( \left( \abs{\approxupsilon{n}{j}}^2 + \abs{\approxxi{n}{j}}^2 \right)  \abs{\approxupsilon{n}{j} - \approxxi{n}{j}}^2  \right) \\
        &\leq C \sum_{j=1}^{J} \expect \left(  \abs{\approxupsilon{n}{j} - \approxxi{n}{j}}^4  \right)^{1/2}
        \leq C \e^{- 2 (M+1)\frac{\Delta}{\delta^2}},
    \end{align*}
which is the required bound.

\bibliographystyle{plain}
\bibliography{main}
\end{document}